\documentclass[12pt]{amsart} 

\usepackage[utf8]{inputenc} 

\usepackage{geometry} 
\usepackage{graphicx} 

\usepackage{booktabs} 
\usepackage{array} 
\usepackage{paralist} 
\usepackage{verbatim} 
\usepackage{subfig} 

\usepackage{amsmath}

\newtheorem{thm}{Theorem}
\newtheorem{lem}[thm]{Lemma}
\newtheorem{prop}[thm]{Proposition}
\newtheorem{cor}[thm]{Corollary}
\newtheorem{rem}{Remark}
\newtheorem{mydef}{Definition}

\newtheorem{conj}[thm]{Conjecture}
\newtheorem{notation}{Notation}

\usepackage{hyperref}

\title{Optimal lower bounds for Donaldson's J-functional}
\author{Zakarias Sj\"ostr\"om Dyrefelt}
\email{zsjostro@ictp.it}

\begin{document}

\begin{abstract}
In this paper we provide an explicit formula for the optimal lower bound of Donaldson's J-functional, in the sense of finding explicitly the optimal constant in the definition of coercivity, which always exists 
and takes negative values in general. This constant is positive precisely if the J-equation admits a solution, and the explicit formula has a number of applications. First, this leads to new 
existence criteria for constant scalar curvature K\"ahler (cscK) metrics in terms of Tian's alpha invariant. Moreover, we use the above formula to discuss Calabi dream manifolds and an analogous notion for the J-equation, and show that for surfaces the optimal bound is an explicitly computable rational function which typically tends to minus infinity as the underlying class approaches the boundary of the K\"ahler cone, even when the underlying K\"ahler classes admit cscK metrics. As a final application we show that if the Lejmi-Sz\'ekelyhidi conjecture holds, then the optimal bound coincides with its algebraic counterpart, the set of J-semistable classes equals the closure of the set of uniformly J-stable classes in the K\"ahler cone, and there exists an optimal degeneration for uniform J-stability.
\end{abstract}

\maketitle


\small

\section{Introduction}

\noindent An important question in K\"ahler geometry is existence of canonical metrics on compact K\"ahler manifolds, and a central role is played by the study of properness of energy functionals on the space of K\"ahler metrics. In particular, the work of Mabuchi \cite{MabuchiKenergy1, Mabuchi, Mabuchisymplectic} and many others established a strong connection between existence of constant scalar curvature K\"ahler metrics and properness of the Mabuchi K-energy functional. A milestone of this approach was Tian's properness conjecture (see e.g. \cite{Tian, DR, BDL}) which was eventually proved in a celebrated series of papers \cite{ChenChengI, ChenChengII, ChenChengIII}. More precisely, building on the work of \cite{Calabiextremal,Mabuchisymplectic, Mabuchi, BB, ChenLiPaun, BDL} and  others, X.X. Chen and J. Cheng recently clarified that existence of twisted cscK metrics is equivalent to properness of the twisted Mabuchi K-energy functional defined on the space of K\"ahler potentials on $X$. 

Motivated by the above we now study various aspects of properness of the functionals corresponding to the constant scalar curvature equation and Donaldson's J-equation respectively. 
To state our results, let $(X,\omega)$ be a compact K\"ahler manifold, with discrete automorphism group, and write $$\mathcal{H} := \{\varphi \in \mathcal{C}^{\infty}(X) \; \vert \; \omega_{\varphi} := \omega + \sqrt{-1}\partial \bar{\partial}\varphi > 0 \}$$ for the associated space of K\"ahler potentials. 
Let $\theta$ be an auxiliary K\"ahler form on $X$ and consider Donaldson's functional $\mathrm{E}_{\omega}^{\theta}$ (in the literature often referred to as the J-functional, see Section \ref{Section variational setup} for the precise terminology), whose critical point equation is precisely Donaldson's J-equation
\begin{equation} \label{Equation J intro}
\mathrm{Tr}_{\omega_{\varphi}} \theta = c,
\end{equation}
where $\varphi \in \mathcal{H}$ and $c$ is the only possible topological constant. The J-functional $\mathrm{E}_{\omega}^{\theta}$ can more generally be considered for arbitrary smooth closed $(1,1)$-forms $\theta$, and appears as the pluripotential part in the Chen-Tian decomposition
\begin{equation}
\mathrm{M}_{\omega}^{\theta} = \mathrm{E}_{\omega}^{\theta} + \mathrm{H}_{\omega} 
\end{equation}
of the twisted Mabuchi K-energy functional, see \cite{Chen2000}, thus connecting properness of $\mathrm{E}_{\omega}^{\theta}$ to existence of cscK metrics. The main goal of this paper is to study properness of Donaldson's J-functional, by determining how the numerical \emph{stability threshold}
\begin{equation}\label{Equation properness intro}
\Gamma^{\mathrm{pp}}_{\theta}(\omega) := \sup \{ \delta \in \mathbb{R} \; \vert \; \exists C > 0, \; \mathrm{E}_{\omega}^{\theta}(\varphi) \geq \delta ||\varphi|| - C , \; \forall \varphi \in \mathcal{H}\},
\end{equation}
varies with $\theta$ and $\omega$, where $||\varphi||$ is a suitable norm (see Section \ref{Section variational setup}). Note that $\Gamma^{\mathrm{pp}}_{\theta}(\omega) > 0$ precisely if \eqref{Equation J intro} has a solution, see \cite[Propositions 21 and 22]{CollinsGabor}. It is here useful to warn the reader that much of the literature focuses on determining when there exists a \emph{positive} constant $\delta$ as in \eqref{Equation properness intro}, but then it follows from coercivity of $\mathrm{H}_{\omega}$ and related properness conjectures (see \cite{Tian, CollinsGabor, ChenChengII}) that the J-equation and the cscK equation automatically admit solutions, so estimating the degree to which it is positive would not be very useful. Instead, the main idea of this paper is to understand precisely how negative the constant $\Gamma^{\mathrm{pp}}_{\theta}(\omega)$ is when the J-functional is \emph{not} coercive (due to connections between the J-equation and the cscK equation such an estimate would have numerous potential applications). To make this rigorous, we first note that the set of $\delta \in \mathbb{R}$ that satisfies the condition in \eqref{Equation properness intro} is non-empty, see Proposition \ref{Prop welldef}. Furthermore, the above properness depends only on the associated cohomology classes $[\theta] \in H^{1,1}(X,\mathbb{R})$ and $[\omega] \in \mathcal{C}_X$, and it is natural to expect that the stability threshold $\Gamma^{\mathrm{pp}}_{\theta}(\omega)$ descends to a function in cohomology. This is confirmed in Proposition \ref{Proposition indep of repr} and Remark \ref{Remark cohomology notation}. 

As a first main result we provide an explicit formula for the stability threshold $\Gamma^{\mathrm{pp}}_{\theta}(\omega)$ on unstable compact K\"ahler surfaces, building on the cohomological condition for solvability of the J-equation in \cite{Chen2000}: 

\begin{thm} \label{Thm main intro surfaces small version}
Let $(X,\omega)$ be a compact K\"ahler surface with discrete automorphism group, and $\theta$ an auxiliary K\"ahler form on $X$. If there is no solution $\varphi \in \mathcal{H}$ to the J-equation $\mathrm{Tr}_{\omega_{\varphi}} \theta = c$, then the stability threshold satisfies
\begin{equation} \label{Equation formula surfaces intro}
\Gamma^{\mathrm{pp}}_{\theta}(\omega) = 
2 \frac{\int_X \theta \wedge \omega}{\int_X \omega^2} - \inf \{\delta > 0 \; : \; \delta[\omega] - [\theta] > 0 \}.
\end{equation}
\end{thm} 

\noindent In other words, either there exists a solution to the J-equation \eqref{Equation J intro}, or we can quantify explicitly how far the J-functional is from being proper. This follows because, by \cite{CollinsGabor} and by construction, the above stability threshold is positive if and only if there exists a solution to Donaldson's J-equation with respect to $\theta$ and  $\omega$ (and this is in turn equivalent to convergence of the J-flow \cite{DonaldsonJobservation,SongWeinkove,CollinsGabor}).  

For the applications to the constant scalar curvature problem that we have in mind, we in general need a version of Theorem \ref{Thm main intro surfaces small version} valid for arbitrary smooth closed $(1,1)$-forms $\theta$ on $X$ (in particular we wish to take $[\theta] = -c_1(X)$). To state it we 
introduce a cohomological condition, involving the Seshadri type constant
$$
\mathcal{T}([\theta],[\omega]) := \sup \{ \delta \in \mathbb{R} \; \vert \; [\theta] - \delta[\omega] \geq 0 \},
$$
which is often possible to compute in practice. We then obtain the following generalized version of Theorem \ref{Thm main intro surfaces small version}: 
\begin{thm} \label{Thm main intro surfaces}
Let $(X,\omega)$ be a compact K\"ahler surface with discrete automorphism group, and let $\theta$ be any smooth closed $(1,1)$-form on $X$. 
If moreover $\Gamma_{\theta}^{\mathrm{pp}}(\omega) < \mathcal{T}([\theta],[\omega])$, then \eqref{Equation formula surfaces intro} holds.
\end{thm} 

\noindent We emphasize that if the above hypothesis is not satisfied, then there always
exist solutions to the J-equation (assuming that $\theta$ and $\omega$ are K\"ahler forms, since then $\mathcal{T}([\theta],[\omega]) > 0$ and $\Gamma^{\mathrm{pp}}_{\theta}(\omega) > 0$ implies existence of a solution, by \cite[Proposition 21]{CollinsGabor}). Along the same lines it follows that if $[\theta]$ is nef then the statement improves on Theorem \ref{Thm main intro surfaces small version}, but when $[\theta]$ is not semi-positive, the cohomological constant $\mathcal{T}([\theta],[\omega])$ may be negative. The obvious application is to consider $\theta := -\mathrm{Ric}(\omega)$, when it is well known that the threshold value is related to properness of the K-energy via Tian's alpha invariant $\alpha_X([\omega])$ (see \cite{Tian}), since the latter controls the entropy term $\mathrm{H}_{\omega}$ of the Mabuchi K-energy functional. As an application of Theorem \ref{Thm main intro surfaces} we thus obtain the following existence criterion for cscK metrics: 
\begin{cor} \label{Cor existence criterion intro}
Let $(X,\omega)$ be a compact K\"ahler surface with discrete automorphism group. Then $X$ admits a constant scalar curvature K\"ahler metric in $[\omega] \in H^{1,1}(X,\mathbb{R})$ if the numerical conditions
\begin{enumerate}
\item 
$$
-2 \frac{\int_X \mathrm{Ric}(\omega) \wedge \omega}{\int_X \omega^2} - \inf \{\delta : -c_1(X) - \delta[\omega] \leq 0 \} > -\frac{3}{2}\alpha_X([\omega])
$$
\item
$$
\mathcal{T}(-c_1(X),[\omega]) > -\frac{3}{2}\alpha_X([\omega])
$$
\end{enumerate}
are both satisfied.
\end{cor} 

\noindent 
This sufficient condition should be compared to other properness criteria using the alpha invariant, see for instance \cite{FLSW, LiShiYao, Dervanalphainvariant}, as well as \cite[Corollary 1.5]{ChenChengII} and references therein (in particular, the above result implies the criteria of Li-Shi-Yao \cite{LiShiYao} and Dervan \cite{Dervanalphainvariant} in the case of surfaces). Higher dimensional analogues may also be obtained, using Theorem \ref{Thm formula higher dimension intro} below, in which case the obtained criterion a priori appears to be completely different to the previously known ones. 

A main point to emphasize is that the expression in Theorem \ref{Thm main intro surfaces} is often explicitly computable in practice, given a good enough understanding of the K\"ahler cone. For explicit computations, see Section \ref{Section examples}, where we revisit J. Ross' slope unstable examples for products of smooth irreducible projective curves $X = C \times C$ of genus $g \geq 2$, see \cite{Rossproductofcurves}. More generally, consider $[\theta] \in \mathcal{C}_X$, let $a \in \partial \mathcal{C}_X$ be a nef but not K\"ahler class, and let $\omega_t$, $t \in [0,1]$, be any K\"ahler forms satisfying $[\omega_t] := (1-t)a + t[\theta]$,  such that $[\omega_t]$ defines a line segment in $H^{1,1}(X,\mathbb{R})$ joining $[\theta]$ to the boundary of the K\"ahler cone at the point $a$. 
Then the cohomological quantity $\Gamma_{\theta}^{\mathrm{pp}}(\omega_t)$ is well-defined for $t > 0$, and the formula \eqref{Equation formula surfaces intro} reduces to the rational function
$$
\Gamma_{\theta}^{\mathrm{pp}}(\omega_t) = 
2 \frac{\int_X \theta \wedge \omega}{\int_X \omega^2} - t^{-1}
$$
under the same conditions as in Theorem \ref{Thm main intro surfaces}. In fact, it is interesting to note that the typical behaviour seems to be that the stability threshold tends to minus infinity as the underlying K\"ahler class $[\omega]$ approaches the boundary of the K\"ahler cone, see also Example \ref{Example blowup}.

\subsection{A formula in higher dimension for manifolds satisfying the Lejmi-Sz\'ekelyhidi conjecture} 

It is natural to ask about generalizations of Theorem \ref{Thm main intro surfaces} in higher dimension. In order to obtain such results we consider the case of compact K\"ahler manifolds $X$ satisfying the Lejmi-Sz\'ekelyhidi conjecture. In what follows, consider the cohomological constant 
$$
C_{\theta,\omega} := 
n \frac{\int_X \theta \wedge \omega^{n-1}}{\int_X \omega^n},
$$
which coincides with the constant in \eqref{Equation J intro}. 

\begin{conj} \label{Conjecture Collins-Szekelyhidi} \emph{\cite{CollinsGabor, LejmiGabor}} Suppose that $(X,\omega)$ is a compact K\"ahler manifold with $\theta$ an auxiliary K\"ahler form on $X$. Then there exists a solution to equation \eqref{Equation J intro} if and only if 
$$
C_{\theta,\omega} \int_V \omega^p > p \int_V \theta \wedge \omega^{p-1}
$$
for every subvariety $V \subset X$ of dimension $p \leq n -1$.
\end{conj} 

\noindent This conjecture was proven for toric manifolds in \cite{CollinsGabor}. Assuming the Lejmi-Sz\'ekelyhidi conjecture we now obtain the following explicit formula generalizing the one for surfaces (the right hand side can in particular be seen to be finite, see Lemma \ref{Lemma inf is finite}): 

\begin{thm} \label{Thm formula higher dimension intro}
Suppose that $X$ is a compact K\"ahler manifold such that the Lejmi-Sz\'ekelyhidi conjecture holds \emph{(e.g. $X$ toric)}, and suppose that  $([\theta], [\omega]) \in H^{1,1}(X,\mathbb{R}) \times \mathcal{C}_X$ with $\Gamma_{\theta}^{\mathrm{pp}}(\omega) < \mathcal{T}([\theta],[\omega]).$ Then 
the stability threshold satisfies
\begin{equation} \label{Equation higher dim formula intro}
\Gamma^{\mathrm{pp}}_{\theta}(\omega) = 
\inf_{V} \frac{C_{\theta,\omega} \int_V \omega^p - p \int_V \theta \wedge \omega^{p-1}}{{(n-p) \int_V \omega^p}},
\end{equation}
where the infimum is taken over all subvarieties $V \subset X$ of dimension $p \leq n - 1$. 
Moreover, the infimum is then achieved by a subvariety $V_{min} \subset X$.
\end{thm}

\begin{rem} 
While this paper was already in preparation a proof of a uniform version of the Lejmi-Sz\'ekelyhidi conjecture was made available by G. Chen \cite{GaoChen}. 
Following their result, we can in fact show that the first part of Theorem \ref{Thm formula higher dimension intro} holds for all compact K\"ahler manifolds with discrete automorphism group, without additional assumptions. Note however that the stronger (non-uniform) statement of the Lejmi-Sz\'ekelyhidi conjecture is needed to guarantee that the infimum is achieved in general, see Remark \ref{Remark inf achieved}.
\end{rem}

\noindent Note that if $X$ is toric then it suffices to test for the (finite number of) toric subvarieties. In general, however, it is helpful to know that the infimum above is in fact achieved, bringing us closer to actually applying the formula for the optimal lower bound in practice.
Characterizing the minimizing subvariety $V_{min} \subset X$ in practice is an important question, related to the circle of ideas surrounding 'optimal degenerations' (geodesic rays or test configurations) in the literature on the Yau-Tian-Donaldson conjecture, which links existence of constant scalar curvature K\"ahler metrics to the algebro-geometric K-stability notion. 

\subsection{Examples and explicit characterization of J-stable classes on surfaces}

A further motivation for the present project is its applications the constant scalar curvature problem, in particular to the discussion on \emph{Calabi dream manifolds}, revived in the recent work of X.X. Chen and J. Cheng \cite{ChenChengII}. 
To put this into context, consider first the case of a compact K\"ahler manifold $X$ with $-c_1(X) > 0$. By analogy with the K\"ahler-Einstein problem and the Aubin-Yau theorem, it was initially suspected that the K-energy should always be proper. This was however shown by J. Ross to be false, by providing examples of products of smooth curves $X = C \times C$ of genus $g \geq 2$ where the twisted K-energy functional is not always proper, see \cite{Rossproductofcurves}. On the other hand, it was shown by Chen-Cheng \cite{ChenChengII}, building on Song-Weinkove \cite{SongWeinkove}, that if $X$ is a compact K\"ahler surface of general type that admits no curve of negative self-intersection, then the energy part of the K-energy is in fact proper over every K\"ahler class. Even though the original expectation of Calabi that every compact K\"ahler manifold is a Calabi Dream Manifold is not correct, it would be important to clarify how ``far'' his vision is from being true. Generally, one would like to understand the picture of the K\"ahler cone; what classes satisfy what stability notions, and what are their relation. Moreover, note that if we can understand precisely what minimal models are Calabi dream manifolds, then 
it would significantly improve our understanding also of constant scalar curvature polarizations on arbitrary compact K\"ahler surfaces, using the blow up results of Arezzo-Pacard \cite{ArezzoPacard}.

As an application of the continuity that follows from the formula in Theorem \ref{Thm main intro surfaces} we now point out a generalization of a result of X.X. Chen and J. Cheng on Calabi dream surfaces. They showed, building on an observation of Donaldson \cite{DonaldsonJobservation}, that every compact K\"ahler manifold with $c_1(X) < 0$ and no curves of negative self-intersection are Calabi dream manifolds. In particular, for such manifolds the cone $\mathrm{Big}_X$ of big $(1,1)$-cohomology classes  equals the K\"ahler cone. In fact, we observe that if $\theta$ is a fixed K\"ahler form, then the $J_{\theta,\omega}$-equation \eqref{Equation J intro} admits a solution in every K\"ahler class $[\omega]$ (for every $\omega \in [\omega]$) if and only if $\mathrm{Big}_X = \mathcal{C}_X$. Moreover, if $\mathrm{Big}_X \neq \mathcal{C}_X$, then it is possible to determine precisely which K\"ahler classes admit a solution: 
 
\begin{thm} \label{Thm application perfect intro}
Let $X$ be a compact K\"ahler surface with discrete automorphism group, with $\eta$ a smooth closed $(1,1)$-form on $X$ such that either $[\eta] \in \mathcal{C}_X$ or $[\eta] = 0$. If $-c_1(X) + [\eta] \geq 0$, then the following holds:   
\begin{enumerate}
\item If $\mathrm{Big}_X = \mathcal{C}_X$, then the $J_{\theta,\omega}$-equation admits a solution for every pair of K\"ahler forms $(\theta,\omega)$ on $X$, and there exists an $\eta$-twisted cscK metric in every K\"ahler class $[\omega] \in \mathcal{C}_X$. 
\item If $\mathrm{Big}_X \neq \mathcal{C}_X$, let $\theta$ be a fixed K\"ahler form on $X$, and suppose that $a \in \partial \mathcal{C}_X$ is a nef but not K\"ahler class which is normalized such that $a^2 = [\theta]^2$. Let $\omega_t$ be any K\"ahler forms on $X$ such that $[\omega_t] := (1-t)a + t[\theta]$ for $t \in (0,1)$. Then the $\mathrm{J}_{\theta,\omega}$-equation \eqref{Equation J intro} admits a solution  
precisely if $[\omega]$ belongs to the subcone 
 $
 \{ \lambda [\omega_t]: \lambda > 0, \; t \in (1/2,1] \} \subset \mathcal{C}_X.
 $
\end{enumerate}   

\end{thm}

\begin{rem} 
\noindent Taking $[\eta] = 0$, the statement $(1)$ extends \cite[Corollary 1.7]{ChenChengII} to the case when $-c_1(X)$ is nef.
\end{rem}

\noindent 
It is moreover interesting to note the following consequence of the explicit formula in Theorem \ref{Thm main intro surfaces small version}, which gives insight about the geometry of the set of J-stable K\"ahler classes: 

\begin{cor}
For each fixed $[\theta] \in \mathcal{C}_X$ the set of K\"ahler classes $[\omega] \in \mathcal{C}_X$ that admits a solution to \eqref{Equation J intro} is open, connected and star convex. 
\end{cor}

\noindent Finally, let $X$ be a Calabi dream surface, and consider $\pi: Y_p \rightarrow X$ the blowup of $X$ at a point $p \in X$. Since blowups admit an exceptional curve, it follows from the above and Arezzo-Pacard \cite{ArezzoPacard} that in any such example there are K\"ahler classes on $X$ that are K-stable but not J-stable. A similar phenomenon can be illustrated concretely in the examples studied by J. Ross in \cite{Rossproductofcurves}, see Section \ref{Section examples}. 

\medskip

\subsection{A formula for algebraic thresholds and applications} The above results can also be proven on the side of algebraic thresholds, related to uniform J-stability, essentially with the same techniques. We state our results using the generalized formalism for test configurations that was introduced in \cite{SD1,DervanRoss}, which yields a J-stability notion which in the projective case coincides with that of Lejmi-Sz\'ekelyhidi \cite{LejmiGabor}. In this terminology, consider the norm of test configuration given by 
$$
||(\mathcal{X},\mathcal{A})|| := \lim_{t \rightarrow +\infty} t^{-1} \mathrm{E}_{\omega}^{\omega}(\varphi_t),
$$
where $(\mathcal{X},\mathcal{A})$ is a test configuration for $(X,[\omega])$ in the sense of \cite{SD1, DervanRoss}. Furthermore, we may define an \emph{algebraic stability threshold} by 
$$
\Delta^{\mathrm{pp}}_{\theta}(\omega) := \sup \{ \delta \in \mathbb{R} \; \vert \; \mathbf{E}^{\theta}_{\omega}(\mathcal{X},\mathcal{A}) \geq \delta ||(\mathcal{X},\mathcal{A})|| \} 
$$
$$
= \inf_{||(\mathcal{X},\mathcal{A})|| = 1} \mathbf{E}^{\theta}_{\omega}(\mathcal{X},\mathcal{A}),
$$
see Section \ref{Section optimal degeneration} for definitions and details. We then note that the optimal lower bound in Theorem \ref{Thm main intro surfaces small version} is in fact achieved also on the side of the non-Archimedean Donaldson J-functional along test configurations: 

\begin{thm} \label{Thm thresholds coincide intro}
Under the conditions of Theorem \ref{Thm formula higher dimension intro} 
the analytic and algebraic stability thresholds coincide, i.e.
$$ 
\Delta^{\mathrm{pp}}_{\theta}(\omega) = \Gamma_{\theta}^{\mathrm{pp}}(\omega).
$$
In particular, we have the formula
$$
\Delta^{\mathrm{pp}}_{\theta}(\omega) = 
\inf_{V} \frac{C_{\theta,\omega} \int_V \omega^p - p\int_V \theta \wedge \omega^{p-1}}{{(n-p) \int_V \omega^p}},
$$ 
where the infimum is taken over all subvarieties $V \subset X$ of dimension $p \leq n - 1$.
\end{thm}

\noindent In other words, the infimum is realized by the test configurations given by the degeneration to the normal cone of $V \subset X$ (and it is therefore not surprising that the analytic and algebraic thresholds coincide, in view of Theorem \ref{Thm formula higher dimension intro}). 
We moreover note that in view of \cite{SD3} it follows that the restrictions of the stability threshold functions $\Gamma_{\theta}^{\mathrm{pp}}(\omega)$ and $\Delta^{\mathrm{pp}}_{\theta}(\omega)$ to the open subset 
$$
\mathcal{V}_X := \{([\theta],[\omega]) \in H^{1,1}(X,\mathbb{R}) \times \mathcal{C}_X \; : \; \Gamma^{\mathrm{pp}}_{\theta}(\omega) < \mathcal{T}([\theta],[\omega]) \}
$$ 
of $H^{1,1}(X,\mathbb{R}) \times \mathcal{C}_X$ are upper semi-continuous. 
Since it is strongly expected that $\Gamma_{\theta}^{\mathrm{pp}}(\omega)$ is lower semi-continuous, this would lead to a proof of continuity of both thresholds, valid also in higher dimension. As previously noted, in the case of surfaces the thresholds are even rational functions. As a consequence we can therefore confirm in this case that the closure of the set of J-stable classes equals the set of J-semistable classes, see Corollary \ref{Cor closure}. 

\subsection{Organization of the paper } 
In Section 2 we introduce analytic preliminaries and definitions of energy functionals and stability thresholds. In Section 3 we establish a variety of properties of the introduced stability thresholds and prove several of our main results: The explicit formula for surfaces is proven in Section \ref{Section main result surfaces} (see Theorems \ref{Theorem main for surfaces} and \ref{Thm main surfaces big version}). The main results in higher dimension is proven in Section \ref{Section higher dim}. In Section \ref{Section examples} we discuss examples and prove Theorem \ref{Thm application perfect intro} (see Theorem \ref{Theorem application 3}). We also discuss applications to twisted constant scalar curvature metrics, and the existence criterion for cscK metrics (Corollary \ref{Cor existence criterion intro}) is proven in Corollary \ref{Corollary existence of cscK criterion}. Finally, in Section \ref{Section optimal degeneration} we study analogous algebraic stability thresholds and prove Theorem \ref{Thm thresholds coincide intro}.

We caution the reader that certain statements in Sections 3-5 should be interpreted from the point of view of cohomology, as explained in Remark \ref{Remark cohomology notation} and Notation \ref{Remark notation}. 

\bigskip

\section{Preliminaries}

\bigskip

\subsection{Variational setup} \label{Section variational setup}
We employ the standard variational setup that is frequently used throughout the K\"ahler geometry literature. To introduce our notation, let $(X,\omega)$ be a compact K\"ahler manifold of complex dimension $n \geq 2$ and write $[\omega]\in H^{1,1}(X,\mathbb{R})$ for the associated K\"ahler class. Let $$V_{\omega} := \int_X \frac{\omega^n}{n!}$$ be the K\"ahler volume of $(X,\omega)$. Let $\mathrm{Ric}(\omega)$ be the Ricci curvature form, normalized such that $[\mathrm{Ric}(\omega)] = c_1(X)$, and write $S(\omega) := \mathrm{Tr}_{\omega}\mathrm{Ric}(\omega)$ for the scalar curvature of $(X,\omega)$.  
Denote the automorphism group of $X$ by $\mathrm{Aut}(X)$ and its connected component of the identity by $\mathrm{Aut}_0(X)$. Write $\mathcal{C}_X \subset H^{1,1}(X,\mathbb{R})$ for the cone of K\"ahler cohomology classes on $X$. Let $\overline{\mathcal{C}}_X$ be the nef cone, $\partial \mathcal{C}_X$ its boundary, and let $\mathrm{Big}_X$ be the cone of big $(1,1)$-classes on $X$.  

We write $(\mathcal{H},d_1)$ for the space of K\"ahler potentials on $X$ endowed with the $L^1$-Finsler metric $d_1$, and denote by $(\mathcal{E}^1,d_1)$ its metric completion (see \cite{Darvas14, Darvas15, Darvassurvey, Darvas, BBGZ, BDLconvexity} and references therein). Write $\mathrm{PSH}(X,\omega) \cap L^{\infty}(X)$ for the space of bounded $\omega$-psh functions on $X$.

Now consider $\varphi \in \mathrm{PSH}(X,\omega) \cap L^{\infty}(X)$. We may define well-known energy functionals
$$
\mathrm{I}_{\omega}(\varphi) := \frac{1}{V_{\omega}n!} \int_X \varphi (\omega^n - \omega_{\varphi}^n)
$$ 
$$
\mathrm{J}_{\omega}(\varphi) = \frac{1}{V_{\omega}n!}\int_X \varphi \omega^n - \frac{1}{V_{\omega}(n+1)!} \int_X \varphi \sum_{j = 0}^n \omega^j \wedge \omega_{\varphi}^{n-j}
$$
$$
\mathrm{E}_{\omega}^{\theta}(\varphi) := \frac{1}{V_{\omega}n!} \int_X \varphi \sum_{j = 0}^{n-1} \theta \wedge \omega^j \wedge \omega_{\varphi}^{n-j-1} - \frac{1}{V_{\omega}(n+1)!}\int_X \varphi \underline{\theta} \sum_{j = 0}^{n} \omega^j \wedge \omega_{\varphi}^{n-j}   
$$
where $\theta$ is any smooth closed $(1,1)$-form on $X$ and $\underline{\theta}$ is the topological constant given by 
$$
\underline{\theta} := 
\frac{\int_X \theta \wedge \frac{\omega^{n-1}}{(n-1)!}}{\int_X \frac{\omega^n}{n!}}.
$$ 
By the Chen-Tian formula \cite{Chen2000} the K-energy functional can be written as the sum of an energy/pluripotential part and an entropy part as
\begin{equation} \label{Equation K-energy}
\mathrm{M}_{\omega} = \mathrm{E}_{\omega}^{-\mathrm{Ric}(\omega)} + \mathrm{H}_{\omega}
\end{equation}
where 
$$
\mathrm{H}_{\omega}(\varphi) := \frac{1}{V_{\omega}n!} \int_X \log \left( \frac{\omega_{\varphi}^n}{\omega^n} \right) \omega_{\varphi}^n
$$
is the relative entropy of the probability measures $\omega_{\varphi}^n/V_{\omega}$ and $\omega^n/V_{\omega}$. In particular, it is well known that $\mathrm{H}_{\omega}(\varphi)$ is always non-negative. 

For any given smooth closed $(1,1)$-form $\theta$ on $X$ we also consider the $\theta$-twisted K-energy functional
$$
\mathrm{M}_{\omega}^{\theta} := \mathrm{M}_{\omega} + \mathrm{E}_{\omega}^{\theta}.
$$
In this paper it will be convenient to measure properness of the K-energy against the functional
$$
(\mathrm{I}_{\omega} - \mathrm{J}_{\omega})(\varphi) = \frac{1}{V_{\omega}(n+1)!} \int_X  \varphi \sum_{j=0}^n \omega^j \wedge \omega_{\varphi}^{n-j} - \frac{1}{V_{\omega}n!} \int_X \varphi \omega_{\varphi}^n
$$
rather than against the usual Aubin $\mathrm{J}$-functional or the $d_1$-distance introduced in \cite{Darvas14}. The following is a standard definition: 

\begin{mydef} \label{Definition properness}
Let $\mathrm{F}: \mathcal{H} \rightarrow \mathbb{R}$ be any of the above considered energy functionals. We then say that $\mathrm{F}$ is coercive if 
$$
\mathrm{F}(\varphi) \geq \delta (\mathrm{I}_{\omega} - \mathrm{J}_{\omega})(\varphi) - C
$$
for some $\delta, C > 0$ and all $\varphi \in \mathcal{H}$.
\end{mydef}

\noindent It is well known that the resulting coercivity notion is equivalent to the other commonly seen coercivity notions in the literature. Indeed, the $(\mathrm{I}_{\omega} - \mathrm{J}_{\omega})$-functional is comparable to $\mathrm{J}_{\omega}$, since we have
$$
\frac{1}{n} \mathrm{J}_{\omega} \leq \mathrm{I}_{\omega} - \mathrm{J}_{\omega} \leq n \mathrm{J}_{\omega},
$$
see \cite[Lemma 6.19, Remark 6.20]{Tianlecturenotes}. The above functionals are moreover comparable to $d_1(0,\varphi)$ (see e.g. \cite{Darvassurvey}) and in addition we have the following relationship between $\mathrm{I}_{\omega} - \mathrm{J}_{\omega}$ and the $E_{\omega}^{\theta}$-functionals: 

\begin{prop} \label{Proposition norm equals twisted} For any compact K\"ahler manifold $(X,\omega)$, and any $\varphi \in \mathcal{H}$, we have
$$
\mathrm{E}_{\omega}^{\omega}(\varphi) = (\mathrm{I}_{\omega} - \mathrm{J}_{\omega})(\varphi).
$$
\end{prop}

\begin{proof}
In this situation we have
$$
\underline{\omega} = n,
$$
and hence rearranging terms yields
\begin{equation*}
\begin{split}
\mathrm{E}_{\omega}^{\omega}(\varphi) & = \frac{1}{V_{\omega}n!} \int_X \varphi \sum_{j = 0}^{n-1} \omega^{j+1} \wedge \omega_{\varphi}^{n-j-1} - \frac{n}{V_{\omega}(n+1)!}\int_X \varphi \sum_{j = 0}^{n} \omega^j \wedge \omega_{\varphi}^{n-j}  \\
& = \frac{n+1}{V_{\omega}(n+1)!} \int_X \varphi \sum_{l = 1}^{n} \omega^{l} \wedge \omega_{\varphi}^{n-l} - \frac{n}{V_{\omega}(n+1)!}\int_X \varphi \sum_{l = 0}^{n} \omega^l \wedge \omega_{\varphi}^{n-l}  \\ 
& = \frac{1}{V_{\omega}(n+1)!} \int_X  \varphi \sum_{l=1}^n \omega^l \wedge \omega_{\varphi}^{n-l} - \frac{n}{V_{\omega}(n+1)!} \int_X  \varphi \omega_{\varphi}^n \\
& = \frac{1}{V_{\omega}(n+1)!} \int_X  \varphi \sum_{l=0}^n \omega^l \wedge \omega_{\varphi}^{n-l} - \frac{1}{V_{\omega} n!} \int_X  \varphi \omega_{\varphi}^n  \\
& = (\mathrm{I}_{\omega} - \mathrm{J}_{\omega})(\varphi).\\
\end{split}
\end{equation*}
\end{proof}

\noindent The above picture is introduced with the following deep result in mind, relating coercivity to existence of constant scalar curvature K\"ahler metrics:

\begin{thm} \label{Theorem ChenCheng main} \emph{(Main theorem of \cite{ChenChengII}, and \cite{BDL})} Let $(X,\omega)$ be a compact K\"ahler manifold and $\theta$ a K\"ahler form on $X$. Then there exists a $\theta$-twisted cscK metric in $[\omega] \in H^{1,1}(X,\mathbb{R})$, i.e. a solution $\varphi \in \mathcal{H}$ to the equation
\begin{equation}\label{Equation twisted cscK}
-\mathrm{Ric}(\omega_{\varphi}) \wedge \omega_{\varphi}^{n-1} + \theta \wedge \omega_{\varphi}^{n-1} = \lambda_{\theta,\omega} \omega_{\varphi}^n
\end{equation}
$$
\lambda_{\theta,\omega} = \frac{(-c_1(X) + [\theta]) \cdot [\omega]^{n-1}}{[\omega]^n}
$$
if and only if $\mathrm{M}_{\omega}^{\theta}$ is coercive. 
\end{thm}

\noindent In view of this result we can reduce our study to understanding coercivity of the relevant (twisted) energy functionals. For the closely related J-equation, introduced by Donaldson in \cite{DonaldsonJequation}, motivated by \cite{UhlenbeckYau}, we also have equivalence between existence and coercivity \cite{CollinsGabor, SongWeinkove}:

\begin{thm} \label{CollinsGabor properness thm} \emph{(\cite{SongWeinkove, CollinsGabor})}
Let $(X,\omega)$ be a compact K\"ahler manifold and $\theta$ a K\"ahler form on $X$. Then there exists a solution $\varphi \in \mathcal{H}$ to the equation
\begin{equation} \label{Equation J}
n\theta \wedge \omega_{\varphi}^{n-1} = C_{\theta,\omega}\omega_{\varphi}^n \; , \; \; C_{\theta,\omega} = n\frac{[\theta] \cdot [\omega]^{n-1}}{[\omega]^n},
\end{equation}
if and only if $\mathrm{E}_{\omega}^{\theta}$ is coercive.
\end{thm}

\medskip

\noindent The approach of this paper is to study the 'optimal constant' $\delta \in \mathbb{R}$ in the definition of coercivity (Definition \ref{Definition properness}). More precisely, we consider the quantities
$$
\Gamma_{\theta}(\omega) := \sup \{ \delta \in \mathbb{R} \; \vert \; \exists C > 0, \; \mathrm{M}_{\omega}^{\theta}(\varphi) \geq \delta (\mathrm{I}_{\omega} - \mathrm{J}_{\omega})(\varphi) - C , \; \forall \varphi \in \mathcal{H} \}
$$
and
$$
\Gamma_{\theta}^{\mathrm{pp}}(\omega) := \sup \{ \delta \in \mathbb{R} \; \vert \; \exists C > 0, \; \mathrm{E}_{\omega}^{\theta}(\varphi) \geq \delta (\mathrm{I}_{\omega} - \mathrm{J}_{\omega})(\varphi) - C , \; \forall \varphi \in \mathcal{H} \}
$$
\\
\noindent which are related to existence of twisted cscK metrics and solutions to equation \eqref{Equation J} respectively. We refer to the above quantities as the \emph{$\theta$-twisted stability threshold} and the \emph{pluripotential stability threshold} respectively, and use the special shorthand $\Gamma(\omega) := \Gamma_0(\omega)$ for the threshold corresponding to the (untwisted) K-energy. Note in particular that the above quantities are well-defined, namely that there always exists a candidate for the supremum in the above definitions (see Section \ref{Section threshold properties} below for a proof of this). Finally, we emphasize that the above stability thresholds are positive if and only if the corresponding functionals are proper, and this is in turn related to existence of solutions to equations \eqref{Equation twisted cscK} and \eqref{Equation J}, provided that $\theta$ and $\omega$ are K\"ahler forms on $X$. The point of view of the remaining sections of this paper is however to use that properness of functionals can be considered even when $\theta$ is taken to be any smooth closed $(1,1)$-form on $X$. 

\bigskip

\section{Properties of stability thresholds and proof of main results}\label{Section threshold properties}

\noindent In this section we establish a number of fundamental properties of stability thresholds $\Gamma_{\theta}^{\mathrm{pp}}(\omega)$ on compact K\"ahler manifolds. In particular we underline that the above stability thresholds are well-defined, and depend only on the underlying cohomology classes. 

\subsubsection{Finiteness and independence of representatives}
First we check that the stability thresholds are well-defined for pairs $(\theta,\omega)$ where $\omega$ is a K\"ahler form on $X$ and $\theta$ is any smooth closed $(1,1)$-form on $X$. More precisely, we check that the thresholds $\Gamma_{\theta}(\omega)$ and $\Gamma^{\mathrm{pp}}_{\theta}(\omega)$ are finite real numbers, that is, the set of constants $\delta \in \mathbb{R}$ such that $\mathrm{E}_{\omega}^{\theta} \geq \delta(\mathrm{I}_{\omega} - \mathrm{J}_{\omega}) - C$ is always non-empty. 

\begin{prop} \label{Prop welldef} For each K\"ahler form $\omega$ on $X$ and each smooth closed $(1,1)$-form $\theta$ on $X$, the threshold values $\Gamma_{\theta}(\omega)$ and $\Gamma_{\theta}^{\mathrm{pp}}(\omega)$ are well-defined and finite, i.e. $$\Gamma_{\theta}(\omega) \in \mathbb{R}, \; \; \Gamma_{\theta}^{\mathrm{pp}}(\omega) \in \mathbb{R}.$$
\end{prop}

\begin{proof}
It is enough to show that the threshold value cannot attain $-\infty$, as long as $\omega$ is K\"ahler. To see this we may without loss of generality assume that 
$$
\int_X  \sum_{j = 0}^n \varphi \omega^j \wedge \omega_{\varphi}^{n-j} = 0.
$$
We can always find a constant $C > 0$ such that $-C\omega \leq \theta \leq C\omega$. Then we have the estimate
$$
|\mathrm{E}_{\omega}^{\theta}(\varphi)| \leq \frac{1}{V_{\omega} n!} \int_X \sum_{j=1}^n |\varphi| \omega^j \wedge \omega_{\varphi}^{n-j} \leq Cd_1(0,\varphi/2),
$$
and (up to changing $C$ if necessary) we have
$$ 
Cd_1(0,\varphi/2) \leq Cd_1(0,\varphi) \leq C(\mathrm{I}-\mathrm{J})(\varphi).
$$
Here we have used \cite[Lemma 3.33]{Darvassurvey} and \cite[Theorem 3]{Darvas15} for the second and third inequalities, and finally the last step can be justified by the well-known double inequality
$$
\frac{1}{C}d_1(0,\varphi) \leq (\mathrm{I}-\mathrm{J})(\varphi) \leq Cd_1(0,\varphi),
$$
see \cite[Remark 6.3]{Darvas15} and \cite[Proposition 5.5]{DR}. 
In conclusion, there is a constant $C := C(||\varphi||_{\theta}, X) > 0$ such that
$$
\mathrm{E}_{\omega}^{\theta}(\varphi) \geq -C(\mathrm{I}-\mathrm{J})(\varphi) - C
$$ 
for all $\varphi \in \mathrm{PSH}(X,\omega) \cap L^{\infty}(X)$. Since the entropy term of the K-energy \eqref{Equation K-energy} is always non-negative, it follows that both 
$\Gamma_{\theta}(\omega)$
and $\Gamma_{\theta}^{\mathrm{pp}}(\omega)$ are well-defined finite real numbers.
\end{proof}

\noindent We moreover show that the stability thresholds depend only on the cohomology classes $[\theta]$ and $[\omega]$, and not on the individual representatives (as long as $\omega$ is assumed K\"ahler).  The following result is a simple refinement of \cite[Corollary 4.7]{ChenChengII} and \cite[Theorem 1]{CollinsGabor}, such that it encompasses also the possibility of twisting forms that are not necessarily K\"ahler: 

\begin{prop} \label{Proposition indep of repr}
Suppose that $(X,\omega)$ is a compact K\"ahler manifold, and let $\theta_1, \theta_2$ be smooth closed $(1,1)$-forms on $X$, with $[\theta_1], [\theta_2] \in H^{1,1}(X,\mathbb{R})$ the associated $(1,1)$-cohomology classes. If $[\theta_1] = [\theta_2]$ then
$$
\Gamma^{\mathrm{pp}}_{\theta_1}(\omega) = \Gamma^{\mathrm{pp}}_{\theta_2}(\omega)
$$
and
$$
\Gamma_{\theta_1}(\omega) = \Gamma_{\theta_2}(\omega).
$$
\end{prop}

\begin{proof}
By \cite[Corollary 4.7]{ChenChengII} (and \cite[Theorem 1]{CollinsGabor} for the J-equation) the result holds whenever $\theta_1$ and $\theta_2$ are K\"ahler forms representing the same K\"ahler class $[\theta_1] = [\theta_2]$. In general, we fix a K\"ahler representative $\omega \in [\omega]$ and pick $\lambda > 0$ large enough so that $\theta_1 + \lambda\omega > 0$ and $\theta_2 + \lambda\omega > 0$ are both K\"ahler representatives of the class $[\theta] + \lambda[\omega]$. By Lemma \ref{Lemma add omega} and \cite[Corollary 4.7]{ChenChengII} we then have 
$$
\Gamma_{\theta_1}(\omega) = \Gamma_{\theta_1 + \lambda\omega}(\omega) - \lambda = \Gamma_{\theta_2 + \lambda\omega}(\omega) - \lambda = \Gamma_{\theta_2}(\omega), 
$$
and the same argument applies also for $\Gamma^{\mathrm{pp}}$.
\end{proof}

\begin{rem} \label{Remark cohomology notation}
It is moreover straightforward to check that $\Gamma^{\mathrm{pp}}_{\theta}(\omega)$ and $\Gamma_{\theta}(\omega)$ are independent of the choice of K\"ahler form $\omega \in [\omega]$. As a consequence of Propositions \ref{Prop welldef} and \ref{Proposition indep of repr} we may thus naturally view the stability thresholds as well-defined functions 
$$\Gamma: H^{1,1}(X,\mathbb{R}) \times \mathcal{C}_X \rightarrow \mathbb{R}$$ and $$\Gamma^{\mathrm{pp}}: H^{1,1}(X,\mathbb{R}) \times \mathcal{C}_X \rightarrow \mathbb{R}$$
on the level of cohomology. The pairs $([\theta],[\omega]) \in \mathcal{C}_X \times \mathcal{C}_X$ for which the function $\Gamma$ \emph{(resp. $\Gamma^{pp}$)} is positive, are precisely those classes where the corresponding constant scalar curvature equation \emph{(resp. $\mathrm{J}$-equation, see \eqref{Equation J})} is satisfied for some suitable choice of K\"ahler forms $\theta \in [\theta]$ and $\omega \in [\omega]$. Note that forms are needed to discuss the J-equation and its solvability, but the abstract conditions both for solvability and properness depend only on cohomology.
\end{rem}

\noindent 
\begin{notation} \label{Remark notation}
\emph{Motivated by the above it would be natural to emphasize the dependence only on cohomology in our notation, by writing $\Gamma_{[\theta]}([\omega])$ and $\Gamma^{\mathrm{pp}}_{[\theta]}([\omega])$ where $([\theta],[\omega]) \in H^{1,1}(X,\mathbb{R}) \times \mathcal{C}_X$. 
It would then be implicit in the notation that we refer to the numerical quantities which are equal respectively to $\Gamma_{\theta}(\omega)$ and $\Gamma^{\mathrm{pp}}_{\theta}(\omega)$ for all choices of K\"ahler forms $\omega \in [\omega]$ and all (not necessarily K\"ahler) smooth closed $(1,1)$-forms $\theta \in [\theta]$. This notation will sometimes be used in the following sections, but since stability thresholds occur so frequently our arguments we most often omit the brackets in our notation, and let the dependence only on cohomology be understood.}
\end{notation}

\noindent For the sequel it is useful to note also the following simple property: 

\begin{lem} \label{Lemma add omega}
Suppose that $\theta$ is a smooth closed $(1,1)$-form on $X$ and $\omega$ a K\"ahler form on $X$. Let $a,b \in  \mathbb{R}$ with $a \geq 0$. Then
$$
\Gamma_{a\theta + b\omega}(\omega) = a \Gamma_{\theta}(\omega) + b
$$
and
$$
\Gamma^{\mathrm{pp}}_{a\theta + b\omega}(\omega) = a\Gamma^{\mathrm{pp}}_{\theta}(\omega) + b.
$$
\end{lem}

\begin{proof}
First note that the $(a\theta + b\omega)$-twisted $\mathrm{E}_{\omega}^{a\theta + b\omega}$-functional is linear in the twist. In fact, we have
\begin{equation} \label{Equation linear functional}
\mathrm{E}_{\omega}^{a\theta + b\omega} = a\mathrm{E}_{\omega}^{\theta} + b\mathrm{E}_{\omega}^{\omega} = a\mathrm{E}_{\omega}^{\theta} + b(\mathrm{I}_{\omega} - \mathrm{J}_{\omega}),  
\end{equation}
where in the last step we have used Proposition \ref{Proposition norm equals twisted}.
The conclusion follows by taking the infimum in \eqref{Equation linear functional}. 
\end{proof}

\subsection{Piecewise linearity of $\Gamma^{\mathrm{pp}}$ in the twisting form}

As a consequence of Lemma \ref{Lemma add omega} we now make the key observation that 
the threshold function
$
\Gamma_{\theta_s}^{\mathrm{pp}}(\omega)
$
is piecewise linear as we vary the twisting form $\theta_s$ along the straight line $\theta_s := (1-s)\theta +s\omega,$
where $s \in \mathbb{R}$. By Proposition \ref{Proposition indep of repr} the piecewise linearity can moreover be described purely in terms of the cohomology classes of the twisting forms. To describe more precisely in what sense this is true, fix any pair $(\theta,\omega)$ as before, such that $([\theta],[\omega]) \in H^{1,1}(X,\mathbb{R}) \times \mathcal{C}_X$, and consider the two dimensional subset of $H^{1,1}(X,\mathbb{R})$ that is linearly spanned by $[\theta]$ and $[\omega]$.
Decompose this subset into two components
$$
\mathrm{Span}([\theta],[\omega]) = \mathrm{Span}([\theta],[\omega])^+  \cup \mathrm{Span}([\theta],[\omega])^- ,
$$
where $$\mathrm{Span}([\theta],[\omega])^+ := \{ a[\theta] + b[\omega], a \geq 0, b \in \mathbb{R} \}$$ and 
$$\mathrm{Span}([\theta],[\omega])^- := \{ a[\theta] + b[\omega], a \leq 0, b \in \mathbb{R}\}.$$ 

\medskip

\noindent We then make the following key observation:

\begin{lem} \label{Lemma piecewise linear}
Suppose that $\beta_0$ and $\beta_1$ are smooth closed $(1,1)$-forms on $X$ such that either $([\beta_0], [\beta_1]) \in \mathrm{Span}([\theta],[\omega])^+ \times \mathrm{Span}([\theta],[\omega])^+$ or $([\beta_0], [\beta_1]) \in \mathrm{Span}([\theta],[\omega])^- \times \mathrm{Span}([\theta],[\omega])^-$.
Then 
\begin{equation} \label{Equation linearity lemma}
\Gamma^{\mathrm{pp}}_{(1-t)\beta_0 + t\beta_1}(\omega) = (1-t)\Gamma^{\mathrm{pp}}_{\beta_0}(\omega) + t\Gamma^{\mathrm{pp}}_{\beta_1}(\omega),
\end{equation}
for each $t \in [0,1]$. 
\end{lem}

\begin{proof}
Suppose first that $([\beta_0], [\beta_1]) \in \mathrm{Span}([\theta],[\omega])^+ \times \mathrm{Span}([\theta],[\omega])^+$. By Proposition \ref{Proposition indep of repr} we may without loss of generality assume that 
$$
\beta_0 = \lambda_0\theta + \lambda'_0\omega
$$
$$
\beta_1 = \lambda_1\theta + \lambda'_1\omega
$$
for some $\lambda_0, \lambda_1, \lambda'_0, \lambda'_1 \in \mathbb{R}$, with $\lambda_0, \lambda_1 \geq 0$. As a consequence
$$
\beta_t := (1-t) \beta_0 + t\beta_1 = \lambda_t \theta + \lambda'_t \omega,
$$
where
$$
\lambda_t := (1-t)\lambda_0 + t\lambda_1 \; , \; \; \lambda'_t := (1-t)\lambda'_0 + t\lambda'_1,
$$
and since $\lambda_0, \lambda_1 \geq 0$ we also have $\lambda_t \geq 0$ for each $t \in [0,1]$. By Lemma \ref{Lemma add omega} we moreover have 
$$
\mathrm{E}_{\omega}^{\beta_t}(\varphi) = \mathrm{E}_{\omega}^{\lambda_t\theta + \lambda'_t\omega}(\varphi) = \lambda_t \mathrm{E}_{\omega}^{\theta}(\varphi) + \lambda'_t(\mathrm{I} - \mathrm{J})(\varphi)
$$
for each $\varphi \in \mathrm{PSH}(X,\omega) \cap L^{\infty}(X)$. In case $\lambda_t = 0$ it is easy to see that $\Gamma_{[\beta_t]}([\omega]) = \lambda'_t$, which is linear in $t$. If we assume instead that $\lambda_t > 0$ then clearly $$\lambda_t \mathrm{E}_{\omega}^{\theta}(\varphi) + \lambda'_t(\mathrm{I} - \mathrm{J})(\varphi) \geq \delta (\mathrm{I} - \mathrm{J})(\varphi) - C$$ if and only if $$\mathrm{E}_{\omega}^{\theta}(\varphi) \geq \frac{(\delta - \lambda'_t )}{\lambda_t} (\mathrm{I} - \mathrm{J})(\varphi) - C.$$ 
This in turn implies that
$$
\Gamma_{\beta_t}^{\mathrm{pp}}(\omega) =  \lambda_t \Gamma_{\theta}^{\mathrm{pp}}(\omega) + \lambda'_t = \lambda_t \Gamma_{\theta}^{\mathrm{pp}}(\omega) + \lambda'_t \Gamma_{\omega}^{\mathrm{pp}}(\omega),
$$
where in the last step we have used that $\Gamma_{[\omega]}^{\mathrm{pp}}(\omega) = 1$ (see Lemma \ref{Lemma add omega}). Since the obtained expression is clearly linear in $t$ it must follow that 
$$
\Gamma^{\mathrm{pp}}_{(1-t)\beta_0 + t\beta_1}(\omega) = (1-t)\Gamma^{\mathrm{pp}}_{\beta_0}(\omega) + t\Gamma^{\mathrm{pp}}_{\beta_1}(\omega),
$$ 
since two linear functions whose values coincide at two points are equal everywhere. This concludes the first part of the proof.

Finally, in case $([\beta_0], [\beta_1]) \in \mathrm{Span}([\theta],[\omega])^- \times \mathrm{Span}([\theta],[\omega])^-$ it suffices to apply the same exact argument but replacing $\theta$ everywhere with $\theta^* := -\theta$. Indeed, then $([\beta_0], [\beta_1]) \in \mathrm{Span}([\theta^*],[\omega])^+ \times \mathrm{Span}([\theta^*],[\omega])^+$ and we show as before that the left hand side of \eqref{Equation linearity lemma} is linear for $t \in [0,1]$. But the left hand side of \eqref{Equation linearity lemma} coincides with the expression on the right hand side for $t = 0$ and $t = 1$. Hence equality must hold for all $t \in [0,1]$, finishing the proof.
\end{proof}

\noindent If we view the threshold function $\Gamma^{\mathrm{pp}}$ as defined on cohomology classes (in the sense of Remark \ref{Remark cohomology notation}) it is thus linear in the twisting form on each of the components $\mathrm{Span}([\theta],[\omega])^+$ and $\mathrm{Span}([\theta],[\omega])^-,$ and takes the value $1$ when $[\theta] = [\omega]$, i.e. on the intersection of the positive and negative chambers. We may further note that the stability threshold is continuous but not differentiable at that point. 
This completes the picture when varying the twisting form $\theta$ while keeping the underlying K\"ahler form $\omega$ fixed.  The behaviour of the stability threshold functions as we vary $\omega$ is more complicated. This question is treated below, first in the case of surfaces, and then in higher dimension.

\medskip

\subsection{Proof of main results for surfaces}  
\label{Section main result surfaces}
Following the above description of the threshold function $\Gamma^{\mathrm{pp}}_{\theta}(\omega)$ 
when varying the twisting class $[\theta]$ (see Remark \ref{Remark cohomology notation}) we can find an explicit formula for the above function in many situations of interest. In the case of compact K\"ahler surfaces this builds on the well-known criterion of \cite{Chen2000} (see also \cite[Theorem 1.1]{SongWeinkove}). It states that there is a solution $\omega_{\varphi} \in [\omega]$ to the equation
\begin{equation} \label{Equation J dim 2}
2\theta \wedge \omega_{\varphi} = C_{\theta,\omega} \omega_{\varphi}^2
\end{equation}
if and only if the difference of K\"ahler classes
\begin{equation} \label{Equation Song-Weinkove criterion}
C_{\theta,\omega}[\omega] - [\theta] > 0,
\end{equation}
where the cohomological constant 
$$
C_{\theta,\omega} := 2\frac{[\theta] \cdot [\omega]}{[\omega]^2},$$
see \cite{Chen2000}. We will sometimes refer to the above as the $J_{\theta,\omega}$-equation, making reference to the underlying K\"ahler forms.

Using the above criterion it is possible to understand more precisely for which pairs of K\"ahler 
forms $(\theta,\omega)$ the above equation is solvable, by working entirely on the level of cohomology classes (this should be compared with a computation in \cite{DonaldsonJobservation}). For example, fix the auxiliary K\"ahler form $\theta$ on $X$, and assume that $a \in \partial \mathcal{C}_X$ is a nef but not K\"ahler cohomology class. Let $\omega_t$ be any K\"ahler form on $X$ such that  
$$
[\omega_t] := (1-t)a + t[\theta], \; \; t \in (0,1].
$$
This defines a line segment of cohomology classes in the K\"ahler cone (and we caution the reader that we do not specify anything about representatives of $a$ here, we only work with the element $a \in H^{1,1}(X,\mathbb{R})$). It is then clear that \eqref{Equation Song-Weinkove criterion} holds for the pair $(\theta, \omega_t)$ 
precisely if 
\begin{equation} \label{Equation Chen reinterpreted}
C_{\theta,\omega_t} - \inf \{\delta \in \mathbb{R} \; \vert \; [\theta] - \delta[\omega_t] \leq 0 \} > 0 
\end{equation}
We first note the following: 

\begin{lem} \label{Lemma one over t}
Let $\theta$ be a K\"ahler form on $X$ and fix $a \in \mathcal{C}_X$ a K\"ahler class. 
For $t \in (0,1]$, let $\omega_t$ be K\"ahler forms such that 
$
 [\omega_t] := (1-t)a + t[\theta],
$
and write
$$
\sigma([\theta],[\omega_t]) := \inf \{\delta \in \mathbb{R} \; \vert \; [\theta] - \delta[\omega_t] \leq 0 \}
$$
$$
\mathcal{T}([\theta],[\omega_t]) := \sup \{ \delta \in \mathbb{R} \; \vert \; [\theta] - \delta [\omega_t] \geq 0\}. 
$$
\\
Then we have the formula
$$
\sigma([\theta],[\omega_t]) = \frac{1}{\mathcal{T}([\omega_t], [\theta])} = \frac{1}{(1-t)\mathcal{T}(a,[\theta]) + t} 
$$
and if $a \in \bar{\mathcal{C}}_X$ is a nef but not K\"ahler cohomology class, then
$$
\sigma([\theta],[\omega_t]) = \frac{1}{t}.
$$
\end{lem}

\begin{proof}
To streamline notation, let us write $\beta := [\theta]$ and $\gamma_t := [\omega_t]$. It is then straightforward to note that
$$
\sigma(\beta,\gamma_t) := \inf \{\delta > 0 \; : \; \beta - \delta\gamma_t \leq 0 \}= \frac{1}{\sup \{\delta > 0 \; : \; \gamma_t - \delta\beta \geq 0 \}} 
$$
$$
= \frac{1}{\mathcal{T}(\gamma_t,\beta)}.
$$
Moreover
$$
\gamma_t - \delta\beta = (1-t)a - (\delta - t)\beta \geq 0
$$
if and only if
$$
a \geq \frac{\delta - t}{1-t}\beta \geq 0,
$$
using that $t \in (0,1]$ by hypothesis. In other words, the quantity $\mathcal{T}(\gamma_t, \beta)$ is linear, namely
$$
\mathcal{T}(\gamma_t, \beta) = (1-t)\mathcal{T}(a,\beta) + t, 
$$
for $t \in [0,1)$. Finally, if $a \in \partial \mathcal{C}_X$, then $\mathcal{T}(a,\beta) = 0$.  
\end{proof} 

\noindent Combining the above Lemma \ref{Lemma one over t} with \eqref{Equation Chen reinterpreted}, it follows that the condition \eqref{Equation Song-Weinkove criterion} is satisfied for $(\theta,\omega_t)$ if and only if
\begin{equation} \label{Equation R(t)}
R(t) := 2\frac{[\theta] \cdot [\omega_t]}{[\omega_t]^2} - \frac{1}{t} = \frac {([\theta]^2 - a^2)t^2 + 2a^2t - a^2}{t[\omega_t]^2} > 0,
\end{equation}
so $R(t) > 0$ is equivalent to solvability of the $J_{\theta,\omega}$-equation, by \cite{Chen2000}. 
We may immediately note that as long as $a^2 = 0$, then $R(t) > 0$ for all $t \in (0,1]$ and so the $J_{\theta,\omega_t}$-equation admits a solution for all pairs $(\theta,\omega_t)$ as above.
In particular, if $X$ is a compact K\"ahler surface such that the big cone equals the K\"ahler cone, then $a^2 = 0$ for all boundary classes, and hence for every pair of K\"ahler classes $(\theta,\omega)$ on $X$ there is a solution to the $J_{\theta,\omega}$-equation. As a special case of interest, if $[\theta] = -c_1(X) > 0$ it follows from the above and Proposition \ref{Proposition indep of repr} that $X$ admits a cscK metric in every K\"ahler class $[\omega] \in \mathcal{C}_X$ (this should be compared with \cite[Corollary 1.7]{ChenChengII} building on Donaldson's observation \cite{DonaldsonJobservation}). 

On the other hand, if $a^2 > 0$, then it is easy to see that $R(t) < 0$ as $[\omega_t]$ approaches the boundary, i.e. for $t > 0$ small enough. For instance, if we normalize $a \in \partial \mathcal{C}_X$ such that $[\theta]^2 = a^2$, then $R(t) < 0$ precisely for $t < 1/2$. 

\begin{thm} \label{Theorem main for surfaces}
Suppose that $X$ is a compact K\"ahler surface with discrete automorphism group.

\begin{itemize}
\item If $\mathrm{Big}_X = \mathcal{C}_X$, then equation \eqref{Equation J} admits a solution for every pair $(\theta,\omega)$ of K\"ahler forms on $X$. 
\item  If $\mathrm{Big}_X \neq \mathcal{C}_X$, then for every pair $(\theta,\omega)$ of K\"ahler forms on $X$, either
\begin{itemize}
\item The J-equation \eqref{Equation J} admits a solution, or 
\item The stability threshold satisfies
\begin{equation}\label{Equation formula surfaces}
\Gamma_{\theta}^{\mathrm{pp}}(\omega) = 
2 \frac{[\theta] \cdot [\omega]}{[\omega]^2} - \inf \{\delta > 0 \; : \; [\theta] - \delta[\omega] < 0 \}.
\end{equation}
\end{itemize} 
\end{itemize} 
\end{thm}

\begin{proof}
Fix K\"ahler forms $\theta$ and $\omega$ on $X$. Consider the path $\theta_t := (1-t)\theta + t\omega$ where $t \in [0,1]$, and introduce the following shorthand notation
$$
P(t) := 2 \frac{[\theta_t] \cdot [\omega]}{[\omega]^2}, \; \;  Q(t) := \sigma([\theta_t],[\omega]), \; \;  R(t) := \Gamma_{\theta_t}^{\mathrm{pp}}(\omega).
$$
It is immediate to check that $P(t)$ and $Q(t)$ are linear, and it follows from Lemma \ref{Lemma add omega} that so is $R(t)$ (having restricted here to $t \in [0,1]$, in general we only get piecewise linearity). In order to show equality of the lhs and rhs it is thus enough to show that the linear functions $L(t) := P(t) - Q(t)$ and $R(t)$ agree for $t = 1$ and for a point $t < 1$. 
First of all, it is immediate to check that $L(1) = R(1) = 1$. Assuming that $L(t) \leq 0$ for some $t \in [0,1)$, there must in particular exist $t_0 \in [0,1)$ for which $L(t_0) = 0$, by continuity. Since $L(t)$ and $R(t)$ are both continuous functions in $t \in (0,1]$ (see Lemma \ref{Lemma piecewise linear}), and moreover $R(t) > 0$ if and only if $L(t) > 0$ (by the existence criterion \eqref{Equation Chen reinterpreted}), it follows that also $R(t_0) = 0$. Hence $L(t)$ and $R(t)$ are both linear and coincide at two points, so $R(t) = L(t)$ for all $t \in [0,1]$. Finally, the first part is an immediate consequence of \eqref{Equation R(t)}.
\end{proof}

\begin{rem} \label{Remark formula hypothesis}
\emph{The above proof works for any $[\omega] \in \mathcal{C}_X$ such that equation \eqref{Equation J} can \emph{not} be solved for \emph{all} pairs $([\theta],[\omega]) \in \mathcal{C}_X \times \mathcal{C}_X$. One way to characterize this is that the formula holds (at least) whenever $$\Gamma_{\theta}^{pp}(\omega) < \mathcal{T}([\theta],[\omega]) := \sup \{\delta \in \mathbb{R} \; \vert \; [\theta] - \delta[\omega] > 0 \}.$$
If we are not in this case then we already know about existence of cscK metrics as well as convergence of the J-flow, so this is not a serious restriction. On the other hand, this adds to the information of Chen's criterion \cite{Chen2000} in a significant way, since being able to measure precisely 'how negative' the threshold value in the unstable cases is important for applications to the constant scalar curvature equation.}
\end{rem}

\begin{thm} \label{Thm main surfaces big version} \emph{(cf. Theorem \ref{Thm main intro surfaces})}
Suppose that $(X,\omega)$ is a compact K\"ahler surface with discrete automorphism group. If $\theta$ is a smooth closed $(1,1)$-form on $X$ such that 
$\Gamma_{\theta}^{\mathrm{pp}}(\omega) < \mathcal{T}([\theta],[\omega])$, then the stability threshold satisfies 
$$
\Gamma_{\theta}^{\mathrm{pp}}(\omega) = 2 \frac{[\theta] \cdot [\omega]}{[\omega]^2} - \inf \{\delta > 0 \; : \; [\theta] - \delta[\omega] < 0 \}.
$$
\end{thm}

\begin{proof}
The proof is essentially the same as that of Theorem \ref{Theorem main for surfaces}. Indeed, suppose that $\Gamma_{\theta}^{\mathrm{pp}}(\omega) < \mathcal{T}([\theta], [\omega])$. It then follows from Lemma \ref{Lemma add omega} that there exists $t_0 \in [0,1)$ such that $R(t_0) = L(t_0)$ (using the notation of the proof of Theorem \ref{Theorem main for surfaces}). The argument then concludes by noting that $R(1) = L(1)$ and invoking the piecewise linearity (Proposition \ref{Lemma piecewise linear}) as before.
\end{proof}

\noindent For clarity we also write down explicitly a particular case relevant for comparisons with the constant scalar curvature equation on surfaces with ample canonical bundle: 

\begin{cor}
Suppose that $X$ is a compact K\"ahler surface with $c_1(X) < 0$. If $X$ does not admit any solution $\omega_{\varphi} \in [\omega]$ to the J-equation 
$$
-n\mathrm{Ric}(\omega) \wedge \omega_{\varphi} = c \omega_{\varphi}^2,
$$
then 
$$
\Gamma_{-\mathrm{Ric}(\omega)}^{\mathrm{pp}}(\omega) = 
-2 \frac{c_1(X) \cdot [\omega]}{[\omega]^2} - \inf \{\delta > 0 \; : \; -c_1(X) - \delta[\omega] < 0 \}.
$$
\end{cor}

\noindent As a direct consequence of Lemma \ref{Lemma one over t} and Theorem \ref{Theorem main for surfaces} we moreover note that in many cases the stability threshold tends to minus infinity as we approach the boundary of the K\"ahler cone: 

\begin{cor}
Let $X$ be a compact K\"ahler surface. Suppose that $\theta$ is a smooth closed $(1,1)$-form on $X$, and suppose that $a \in \partial\mathcal{C}_X$ with $a^2 > 0$. Let $\omega_t$ be any K\"ahler forms on $X$ such that $[\omega_t] := (1-t)a + t[\theta]$, $t \in [0,1]$. Then there is a uniform constant $C > 0$ such that
$$
\Gamma_{\theta}^{\mathrm{pp}}(\omega_t) \leq C - t^{-1}
$$
for all $t \in (0,1)$.
\end{cor}

\noindent This is elaborated on in the examples of Section \ref{Section examples}.

\medskip

\subsection{A formula in higher dimension for manifolds satisfying the Lejmi-Sz\'ekelyhidi conjecture}\label{Section higher dim}
In this section we discuss natural generalizations of Theorem \ref{Thm main intro surfaces} to higher dimensions, and provide such a result for manifolds satisfying a conjecture of Lejmi-Sz\'ekelyhidi 
(see \cite[Conjecture 1]{LejmiGabor}). In order to state it, recall that 
$$
C_{\theta,\omega} := 
n\frac{\int_X \theta \wedge \omega^{n-1}}{\int_X \omega^n}.
$$   
The Lejmi-Sz\'ekelyhidi conjecture then states the following: 

\begin{conj} \label{Conjecture Collins-Szekelyhidi} \emph{\cite{CollinsGabor, LejmiGabor}} Suppose that $(X,\omega)$ is a compact K\"ahler manifold with $\theta$ an auxiliary K\"ahler form on $X$. Then there exists a solution $\omega_{\varphi} \in [\omega]$ to the equation \eqref{Equation J} if and only if 
$$
C_{\theta,\omega} \int_V \omega^p > p\int_V \omega^{p-1} \wedge \theta > 0
$$
for every subvariety $V \subset X$ of dimension $p \leq n -1$.
\end{conj} 

\noindent The above conjecture was proven for toric manifolds by T. Collins and G. Sz\'ekelyhidi \cite{CollinsGabor}, but is expected to hold for all compact K\"ahler manifolds (in fact, a proof of a uniform version of this conjecture was published by G. Chen \cite{GaoChen} while this paper was already in preparation). In the surface case, the above reduces to a special case of Chen's theorem \cite{Chen2000} (see also \cite[Theorem 1.1]{SongWeinkove}). That can be seen from the Nakai-Moishezon criterion.

Deducing a generalized formula for stability thresholds in higher dimension essentially follows the same idea as in the surface case, exploiting the piecewise linearity in the twisting argument. Indeed, the abstract principle is to find an expression which is linear in the above mentioned argument, which takes the value $1$ when $[\theta] = [\omega]$, and which is strictly positive precisely when the corresponding $J_{\theta,\omega}$-equation \eqref{Equation J} can be solved. In order to state the main result of this section we first need the following lemma: 

\begin{lem} \label{Lemma inf is finite} For any compact K\"ahler manifold $X$ the quantity
$$
\inf_{V} \frac{ C_{\theta,\omega} \int_V \omega^p - p\int_V \omega^{p-1} \wedge \theta}{{(n-p) \int_V \omega^p}}
$$
is finite, where the infimum is taken over all subvarieties $V \subset X$ of dimension $p \leq n-1$. 
\end{lem}

\begin{proof}
Introduce the shorthand notation
$$
L_V([\theta],[\omega]) := \frac{ C_{\theta,\omega} \int_V \omega^p - p\int_V \omega^{p-1} \wedge \theta}{{(n-p) \int_V \omega^p}},
$$
and observe that for any $[\omega] \in \mathcal{C}_X$, and any subvariety $V \subset X$, we have $L_V([\omega],[\omega]) = 1$. For any fixed $V \subset X$ and any pair $(\theta,\omega)$ of K\"ahler forms on $X$, let $\theta_s := (1-s)\omega + s\theta$, with $s \in [0,1]$. Then clearly the function $s \mapsto L_V([\theta_s],[\omega])$ is linear. Moreover, as in Proposition \ref{Prop welldef} there is a uniform constant $C > 0$ such that $$|\mathrm{E}_{\omega}^{\theta}(\varphi)| \leq C\mathrm{E}_ {\omega}^{\omega}(\varphi) - C.$$
It follows from this that $\mathrm{E}_{\omega}^{\theta_s}$ is proper for $s \in [0,1]$ small enough, and hence there must exist a constant $C_{\theta}' > 0$ such that $L_V([\theta_s],[\omega]) > 0$ for all $V \subset X$ and all $s < C'_{\theta}$. But then
$$
L_V([\theta],[\omega]) \geq -C_{\theta}'
$$ 
for all smooth $(1,1)$-forms on $X$ and all $V \subset X$ by linearity (Lemma \ref{Lemma piecewise linear}). In particular it follows that $\inf_V L_V([\theta],[\omega])$ is finite whenever $([\theta],[\omega]) \in H^{1,1}(X,\mathbb{R}) \times \mathcal{C}_X$. 
\end{proof}

\noindent The precise formulation of our main result in higher dimension $n \geq 2$ is then as follows: 

\begin{thm} \label{Thm formula higher dimension}
Suppose that $(X,\omega)$ is a compact K\"ahler manifold and assume that the Lejmi-Sz\'ekelyhidi conjecture holds for $X$ \emph{(e.g. $X$ toric)}. Let $\theta$ be a smooth closed $(1,1)$-form on $X$ such that $\Gamma^{\mathrm{pp}}_{\theta}(\omega) < \mathcal{T}([\theta],[\omega])$. Then we have the formula
$$
\Gamma^{\mathrm{pp}}_{\theta}(\omega) = \inf_{V} \frac{ C_{\theta,\omega} \int_V \omega^p - p\int_V \omega^{p-1} \wedge \theta}{{(n-p) \int_V \omega^p}},
$$ 
where the infimum is taken over all subvarieties $V \subset X$ of dimension $p \leq n - 1$. 
\end{thm}

\noindent In particular, it should be noted that the right hand side is in fact a finite real number, see Lemma \ref{Lemma inf is finite}.  

\begin{rem} \label{Remark inf achieved} 
\emph{In a recent work \cite{GaoChen} which was made available when this note was already in preparation, G. Chen showed a uniform version of Conjecture \ref{Conjecture Collins-Szekelyhidi}. Following their result, the above Theorem \ref{Thm formula higher dimension} can be seen to hold for all compact K\"ahler manifolds with discrete automorphism group, without further restrictions. The proof is the same. 
On the other hand, it should be noted that the assumption that the non-uniform statement in the Lejmi-Sz\'ekeleyhidi conjecture holds is necessary for Theorem \ref{Thm optimal degeneration intro} below. This hypothesis is sometimes implied by the result of G. Chen \cite{GaoChen}, in particular whenever it is enough to test for a finite number of subvarieties.}  
\end{rem}

\begin{proof}[Proof of Theorem \ref{Thm formula higher dimension}]
The proof is very similar to that of Theorem \ref{Theorem main for surfaces}. Consider the path 
$
\theta_s := (1-s)\theta + s\omega,
$
where $s \in [0,1]$. As before it is immediate to check that for each subvariety $V \subset X$ of dimension $p \leq n-1$, the quantities
$$
P_V(s) :=\frac{C_{\theta_s,\omega} \int_V \omega^p - p\int_V \omega^{p-1} \wedge {\theta_s}}{{(n-p) \int_V \omega^p}}, \; \;  Q(s) :=  \Gamma_{\theta_s}^{\mathrm{pp}}(\omega)
$$
are linear for $s \in [0,1]$ (for $Q(s)$ we use Lemma \ref{Lemma piecewise linear} as before). Moreover, we can see that 
$$
L(s) := \inf_V P_V(s)
$$
is also linear for $s \in [0,1]$ (albeit only piecewise linear when we allow  $s > 1$), noting that this holds even if we take the infimum over an infinite number of subvarieties, since we have linearity and $P_V(1) = 1$ for all $V$. With the given normalization $P_V(1) = 1$ for every $V \subset X$ it also follows that $L(1) = Q(1) = 1$. Finally, by hypothesis $[\omega] \in \mathcal{C}_X$ admits no solution to the $\mathrm{J}$-equation, so by Theorem \ref{CollinsGabor properness thm} (see \cite[Proposition 21 and 22]{CollinsGabor}) we have $\Gamma_{\theta}^{\mathrm{pp}}(\omega) \leq 0$, so $Q(s_0) = 0$ for some $s_0 \in [0,1)$. Moreover, whenever Conjecture \ref{Conjecture Collins-Szekelyhidi} holds, also $L(s_0) = 0$ must hold. Hence $L(s)$ and $R(s)$ are both linear and coincide at two points, so $R(s) = L(s)$ for all $s \in [0,1]$.   
This completes the proof.
\end{proof}

\noindent So far this gives an exact formula for the stability threshold, but compared to the case of surfaces (where there is a single cohomological condition that is easily computed) it is considerably less explicit. Also in the toric case $X$ it suffices to test for the (finite number of) toric subvarieties. In general, however, it is helpful to know that the infimum above is in fact achieved, bringing us closer to actually applying the formula for the optimal lower bound in practice: 

\begin{thm} \label{Thm optimal degeneration intro}
Under the same assumptions, the infimum in Theorem \ref{Thm formula higher dimension} is achieved by a subvariety $V_{min} \subset X$, i.e. 
$$
\Gamma^{\mathrm{pp}}_{\theta}(\omega) = \frac{ C_{\theta,\omega} \int_{V_{min}} \omega^p - p\int_{V_{min}} \omega^{p-1} \wedge \theta}{{(n-p) \int_{V_{min}} \omega^p}}.
$$ 
\end{thm}

\begin{proof}
Using the same notation as in the proof of Theorem \ref{Thm formula higher dimension}, let $
\theta_s := (1-s)\theta + s\omega$. By the hypothesis $\Gamma^{\mathrm{pp}}_{\theta}(\omega) < \mathcal{T}([\theta],[\omega])$ it follows that $\Gamma^{\mathrm{pp}}_{\theta_{s_0}}(\omega) = 0$ for some $s_0 < 1$. Note moreover that if Conjecture \ref{Conjecture Collins-Szekelyhidi} holds for $X$, then $\Gamma^{\mathrm{pp}}_{\theta_{s_0}}(\omega) = 0$ if and only if 
$$
C_{\theta_{s_0},\omega} \int_{V_{s_0}} \omega^p - p\int_{V_{s_0}} \omega^{p-1} \wedge \theta_{s_0} = 0
$$
for some subvariety $V_{s_0} \subset X$, so the infimum is achieved for $s = s_0$. By piecewise linearity in the twisting form (Lemma \ref{Lemma piecewise linear}), it however follows that the infimum is realized also for all $s < 1$, in particular for $\Gamma^{\mathrm{pp}}_{\theta}(\omega)$. Setting $V_{min} := V_{s_0}$ finishes the proof.  
\end{proof}

\noindent It is interesting to study further the minimizing subvariety $V_{min} \subset X$ above. Although not much is currently known about how to characterize such minimizing subvarieties, it follows immediately from our techniques that if $V_{\theta,\omega}$ is a subvariety that realizes the infimum for the pair $(\theta,\omega)$, then it depends only on the cohomology classes $[\theta]$ and $[\omega]$ (by Remark \ref{Remark cohomology notation}) and the same subvariety realizes the infimum as we vary $[\theta]$, but keep $[\omega]$ fixed, in the following sense (see Section \ref{Section main result surfaces} for the definition of $\mathrm{Span}([\theta],[\omega])^+$): 
\begin{thm} \label{Thm inf achieved}
Suppose that $V_{min} \subset X$ is a minimizing subvariety for the pair $([\theta],[\omega]) \in H^{1,1}(X,\mathbb{R}) \times \mathcal{C}_X$, i.e. 
$$
\Gamma^{\mathrm{pp}}_{\theta}(\omega) = \frac{ C_{\theta,\omega} \int_{V_{min}} \omega^p - p\int_{V_{min}} \omega^{p-1} \wedge \theta}{{(n-p) \int_{V_{min}} \omega^p}}.
$$ 
Assume moreover that $[\theta] \neq [\omega]$. Then the same subvariety $V_{min}$ achieves the infimum for all pairs $([\theta'],[\omega])$ with $[\theta'] \in \mathrm{Span}([\theta],[\omega])^+ := \{a[\theta] + b[\omega] \; : \; a\geq 0, b \in \mathbb{R}\}$. 
\end{thm}

\begin{proof}
This is an immediate consequence of Proposition \ref{Proposition indep of repr} and linearity of $\Gamma_{\theta}^{\mathrm{pp}}(\omega)$ as we vary $[\theta]$ in $\mathrm{Span}([\theta],[\omega])^+$, see Lemma \ref{Lemma piecewise linear}.
\end{proof}

\begin{rem} \emph{The existence of optimal degenerations is a question of interest in particular in connection with the Yau-Tian-Donaldson conjecture for constant scalar curvature K\"ahler metrics (see e.g. \cite{Donaldsontoric, CDSone, CDStwo, CDSthree, TianYTDconjecture, GaborDatar}). In particular, the above result may be compared with a recent result on existence of minimizing geodesic rays for J-semistable and J-unstable classes in \cite{Mingchen}.}
\end{rem}

\bigskip

\section{Applications part I: The J-equation and Calabi Dream Manifolds}\label{Section applications}
\noindent In this section we continue the discussion in Section \ref{Section main result surfaces} and give first applications and consequences of the explicit formula in Theorem \ref{Thm main surfaces big version}. In particular, we give a more detailed and completely explicit criterion for determining the K\"ahler classes on an arbitrary compact K\"ahler surface that admits a solution to the J-equation. We furthermore discuss applications to Calabi dream manifolds and show that the stability threshold typically tends to minus infinity as the underlying K\"ahler class approaches the boundary of the K\"ahler cone, but nonetheless the entropy compensates and the constant scalar curvature equation can sometimes be solved. We also show that the purely cohomological formula in Theorem \ref{Thm main intro surfaces small version} holds also for the natural definition of algebraic stability threshold corresponding to uniform J-stable. From this we then deduce continuity, and openness of (uniform) J-stability, as well as existence of an optimal 'minimizing' test configuration for uniform J-stability on surfaces.

Note that the list of applications is by no means intended to be exhaustive.

\subsection{Explicit characterization of J-stable classes and applications to Calabi dream manifolds}

\noindent Continuing the discussion in Section \ref{Section main result surfaces}, we are interested in characterizing all compact K\"ahler manifolds for which the J-equation can always be solved (this can be thought of as an analogy to the so called 'Calabi dream manifolds' notion for the constant scalar curvature problem, see \cite{ChenChengII}). We first introduce some terminology: 
\begin{mydef} \label{Definition perfect} 
Let $X$ be a compact K\"ahler manifold such that for every pair $([\theta], [\omega]) \in \mathcal{C}_X \times \mathcal{C}_X$, and every choice of K\"ahler forms $\theta \in [\theta]$ and $\omega \in [\omega]$,  
the $J_{\theta,\omega}$-equation 
\begin{equation}
n\theta \wedge \omega_{\varphi}^{n-1} = C_{\theta,\omega} \omega_{\varphi}^n 
\end{equation}
admits a solution $\omega_{\varphi} \in [\omega]$. Then $X$ is said to be \emph{perfect}.
\end{mydef}

\noindent It is clear that being perfect is a purely cohomological property, due to \cite[Theorem 1]{CollinsGabor}. As a first example, it was observed by Donaldson \cite{DonaldsonJobservation} that any compact K\"ahler surface with $-c_1(X) > 0$ and no curves of negative self-intersection are perfect. The condition on negative curves in particular means that the cone $\mathrm{Big}_X$ of big $(1,1)$-cohomology classes on $X$ coincides with the K\"ahler cone. Along the lines of the argument outlined in Section \ref{Section main result surfaces}, we now point out that these are the only possible examples:

\begin{prop} \label{Theorem application 1} 
A compact K\"ahler surface with discrete automorphism group is perfect if and only if it admits no curves of negative self-intersection.
\end{prop}

\begin{proof}
The statement that a compact K\"ahler surface $X$ admits no curves of negative self-intersection is equivalent to $\mathrm{Big}_X = \mathcal{C}_X$. This is in turn equivalent to every nef but not K\"ahler class $a \in \partial\mathcal{C}_X$ having zero volume $a^2 = 0.$ In other words, it follows from the discussion in Section \ref{Section main result surfaces} that equation \eqref{Equation J} is solvable for every pair $([\theta], [\omega]) \in \mathcal{C}_X \times \mathcal{C}_X$ if and only if $X$ admits no curves of negative self-intersection. 
\end{proof}

\noindent Making the connection with the (twisted) constant scalar curvature problem, we note that perfect surfaces are always Calabi dream surfaces, by properness of the entropy functional (see \cite{Tian, ChenChengII}). When $X$ is a compact K\"ahler surface such that the K\"ahler cone is strictly contained in the big cone, we moreover characterize completely (in terms of their associated cohomology classes) the pairs of K\"ahler forms on $X$ for which the J-equation is solvable: 

\begin{thm} \label{Theorem application 3}
Let $X$ be a compact K\"ahler surface with discrete automorphism group. Suppose that $\eta$ is a smooth closed $(1,1)$-form on $X$ such that either $[\eta] \in \mathcal{C}_X$ or $[\eta] = 0$, and assume that $-c_1(X) + [\eta] \geq 0$. 
\begin{enumerate}
\item If $\mathrm{Big}_X = \mathcal{C}_X$, then the $J_{\theta,\omega}$-equation admits a solution for every pair of K\"ahler forms $(\theta,\omega)$ on $X$, and there exists an $\eta$-twisted cscK metric in every K\"ahler class on $X$. 
\item If $\mathrm{Big}_X \neq \mathcal{C}_X$, let $\theta$ be a fixed K\"ahler form on $X$, and suppose that $a \in \partial \mathcal{C}_X$ is a nef but not K\"ahler class which is normalized such that $a^2 = [\theta]^2$. Let $\omega_t$ be any K\"ahler forms on $X$ such that $[\omega_t] := (1-t)a + t[\theta]$ for $t \in (0,1]$. Then the $\mathrm{J}_{\theta,\omega}$-equation \eqref{Equation J intro} admits a solution 
precisely if $[\omega]$ belongs to the subcone 
 $$
 \mathcal{I}_{\theta} = \{ \lambda [\omega_t]: \lambda > 0, \; t \in (1/2,1] \}. 
 $$
\end{enumerate}   
\end{thm}

\begin{rem} \emph{The connection to twisted constant scalar curvature metrics in $(1)$ uses \cite[Theorem 6.1]{ChenChengII}. 
The statement moreover addresses a discussion on Calabi dream surfaces in \cite{ChenChengII}. 
In particular, taking $[\eta] = 0$, this extends \cite[Corollary 1.7]{ChenChengII} to the case when $-c_1(X)$ is nef.}
\end{rem}

\begin{proof}[Proof of Theorem \ref{Theorem application 3}]
We write the proof using the cohomological shorthand notation introduced in Remark \ref{Remark cohomology notation} and Notation \ref{Remark notation}: First note that Proposition \ref{Theorem application 1} gives the first part of $(1)$. For the remaining part, let $\omega$ be any K\"ahler form on $X$, and let
$\rho := -c_1(X) + [\eta]$. By hypothesis we have $\rho \geq 0$, so $\rho_{\epsilon} := -c_1(X) + [\eta] + \epsilon[\omega] > 0$ for every $\epsilon > 0$. Moreover, since $\mathrm{Big}_X = \mathcal{C}_X$ it follows by Proposition \ref{Theorem application 1} and Lemma \ref{Lemma add omega} that
$$
\Gamma_{\rho}^{\mathrm{pp}}([\omega]) + \epsilon = \Gamma_{\rho_{\epsilon}}^{\mathrm{pp}}([\omega]) > 0
$$ 
for every $\epsilon > 0$. Hence $\Gamma_{\rho}^{\mathrm{pp}}([\omega]) \geq 0$ follows. 
Due to \cite[Theorem 6.1]{ChenChengII} and properness of the entropy functional $\mathrm{H}_{\omega}$ (see \cite{Tian}) this implies that the K\"ahler class $[\omega] \in \mathcal{C}_X$ admits an $\eta$-twisted cscK metric. Since $[\omega] \in \mathcal{C}_X$ was arbitrary, this proves $(1)$.

To prove the second part, suppose that $[\theta] \in \mathcal{C}_X$, $a \in \partial \mathcal{C}_X$ and the volume is normalized such that $[\theta]^2 = a^2$. It was then noted in Section \ref{Section main result surfaces} that $\gamma_t := (1-t)a + t[\theta]$, $t \in (0,1]$ admits a solution to the $J_{\theta,\omega_t}$-equation (for all K\"ahler forms $\theta \in [\theta]$ and $\omega_t \in \gamma_t$) precisely when $1/2 < t \leq 1$. Since solvability is independent of rescaling by a positive constant, it follows that $\{ \lambda \gamma_t: \lambda > 0, \; 1/2 < t \leq 1 \} \subseteq \mathcal{I}_{\theta}.$ On the other hand, if $\gamma \in \mathcal{C}_X$ is an arbitrary K\"ahler class, 
consider the rescaled class $\tilde{\gamma} := \sqrt{a^2/\gamma^2}\gamma$ and apply the same argument as above. This finishes the proof.
\end{proof}

\noindent Recall moreover that it is shown in \cite{ChenChengII} that if $X$ is a compact K\"ahler surface with no curves of negative self-intersection, then $X$ is a Calabi dream surface. If we take $[\eta] = 0$ in Theorem \ref{Theorem application 3} then this can in particular be seen as an extension of the above result to the situation when $-c_1(X)$ is merely nef: 

\begin{cor}
Suppose that $X$ is a compact K\"ahler surface with discrete automorphism group, which satisfies $-c_1(X) \geq 0$ and $\mathrm{Big}_X = \mathcal{C}_X$. 
Then it is a Calabi dream surface. 
\end{cor}

\noindent This is interesting input to a question of X.X. Chen and J. Cheng \cite{ChenChengII} on 'how far the class of Calabi dream surfaces is from the set of minimal surfaces of general type', since indeed we show that there exist Calabi dream surfaces other than those with $-c_1(X) > 0$ and no curves of negative self-intersection (i.e. the ones pointed out in \cite{ChenChengII}). To obtain concrete examples of Calabi dream manifolds, see e.g. the discussion in \cite[Section 2]{ChenChengII}. 

Finally, with applications to properness of the Mabuchi K-energy functional in mind, we consider the case when $c_1(X)$ is not assumed to have a sign. We then note that, regardless of the sign (or lack of sign) of $c_1(X)$, only surfaces with $\mathrm{Big}_X = \mathcal{C}_X$ can be perfect. To make sense of this statement, recall from the discussion in Section \ref{Section threshold properties} (see in particular Proposition \ref{Proposition indep of repr}) that we may naturally extend $$\Gamma^{\mathrm{pp}}: \mathcal{C}_X \times \mathcal{C}_X \rightarrow \mathbb{R}$$ to a functional
$$
\Gamma^{\mathrm{pp}}: H^{1,1}(X,\mathbb{R}) \times \mathcal{C}_X \rightarrow \mathbb{R}, \; \; ([\theta], [\omega]) \mapsto \Gamma^{\mathrm{pp}}_{\theta}(\omega),
$$
i.e. where we allow the twisting form $\theta$ to be any smooth closed $(1,1)$-form on $X$ (note however that we only consider the connection between properness and solvability of the corresponding equations in case the twisting form is K\"ahler, \cite{CollinsGabor, ChenChengII}). 
We then have the following:

\begin{cor} \label{Thm 15}
Let $X$ be a compact K\"ahler surface with discrete automorphism group, and fix any smooth closed $(1,1)$-form $\rho$ on $X$. Suppose that the functional $\mathrm{E}_{\omega}^{\rho}: \mathcal{H} \rightarrow \mathbb{R}$ is coercive for every K\"ahler form $\omega$ on $X$. Then $X$ admits no curve of negative self-intersection. 
\end{cor}

\begin{proof}
Suppose that $X$ does admit a curve of negative self-intersection. It then suffices to show that $\mathrm{E}_{\omega}^{\rho}$ is not coercive for some suitable choice of K\"ahler form $\omega$ on $X$. To see this, first note that if $-c_1(X) > 0$, then this follows from the second part of Theorem \ref{Theorem application 3}. If  $c_1(X) > 0$ then noting that $$\Gamma^{\mathrm{pp}}_{-c_1(X)}(c_1(X)) = -1 < 0$$ is enough to conclude (interpreting the above notation in the sense of Remark \ref{Remark cohomology notation} and Notation \ref{Remark notation}). In case $c_1(X) = 0$ the conclusion is immediate. Finally, it remains to treat the case when the canonical class does not have a sign.  Exploiting the cohomological interpretation of stability thresholds (see Remark \ref{Remark cohomology notation} and Notation \ref{Remark notation}), fix a cohomology class $a \in \partial \mathcal{C}_X$ which is nef but not K\"ahler. Let $\theta$ be a K\"ahler form on $X$ and fix any K\"ahler forms $\omega_t$ on $X$ such that $$[\omega_t] := (1-t)a + t[\theta]$$ for $t \in (0,1]$. Then by Section \ref{Section main result surfaces} the threshold value
$$
\Gamma_{\theta}^{\mathrm{pp}}(\omega_t) \leq C - t^{-1}
$$
where $C > 0$ is a uniform constant. We may moreover assume that $a \in \partial \mathcal{C}_X$ is chosen so that there is a constant $\lambda > 0$, for which $-c_1(X) + \lambda[\omega_t] > 0$ for all the associated K\"ahler classes $[\omega_t]$ and all $t \in (0,1)$. Putting this together and using Lemma \ref{Lemma add omega} it follows that
$$
\Gamma_{-c_1(X)}^{\mathrm{pp}}([\omega_t]) = \Gamma_{-c_1(X) + \lambda [\omega_t]}^{\mathrm{pp}}([\omega_t]) - \lambda < 0
$$
for $t > 0$ small enough. 
\end{proof}

\noindent It is worth emphasizing that if we take $[\rho] = -c_1(X) \in H^{1,1}(X,\mathbb{R})$ (e.g. $\rho := -\mathrm{Ric}(\omega)$), then the exact same proof yields the following:

\begin{cor} \label{Corollary unstable example}  
Let $X$ be a compact K\"ahler surface with discrete automorphism group. If the pluripotential part $\mathrm{E}_{\omega}^{-\mathrm{Ric}(\omega)}$ of the Mabuchi K-energy functional is coercive for every K\"ahler form $\omega$ on $X$, then $X$ admits no curve of negative self-intersection. 
\end{cor} 

\noindent One of the key features to highlight is that even though the threshold associated with the pluripotential term of the K-energy functional tends to minus infinity as we approach the boundary, this may be compensated for by the entropy part (which is always proper \cite{Tian}). At the moment we can only note that, if the entropy term satisfies
$$
\mathrm{H}_{\omega}(\varphi) \geq \delta(\mathrm{I}_{\omega} - \mathrm{J}_{\omega})(\varphi) - C,
$$ 
and 
$$
\mathrm{\Gamma}^{\mathrm{pp}}_{-\mathrm{Ric}(\omega)}(\omega) > -\delta,
$$
then $(X,\omega)$ admits a cscK metric. Clarifying the interaction of the entropy and pluripotential/energy terms seems to be an important question in order to apply the results of this paper to the cscK problem. 

The above strategy can however already be combined with the explicit formula of Theorem \ref{Thm main intro surfaces} to give a criterion for existence of constant scalar curvature K\"ahler metrics using Tian's alpha invariant (see \cite{Tian}) for $(X,\omega)$, which is given by
$$
\alpha_X([\omega]) := \sup \{\alpha > 0 \; : \; \exists C > 0, \int_X e^{-\alpha(\varphi - \sup \varphi)} \omega^n \leq C, \; \forall \varphi \in \mathcal{H}(X,\omega) \}. 
$$
The following consequence of our main results add to the active research in this direction (see e.g. \cite{FLSW, LiShiYao, Dervanalphainvariant}, as well as \cite[Corollary 1.5]{ChenChengII} and references therein):   

\begin{cor} \label{Corollary existence of cscK criterion}
Let $(X,\omega)$ be a compact K\"ahler surface with discrete automorphism group. 
Then $X$ admits a constant scalar curvature K\"ahler metric in $[\omega]$ if the numerical condition
\begin{equation} \label{Equation min equation cscK}
\min \left(- 2 \frac{c_1(X) \cdot [\omega]}{[\omega]^2}  - \sigma(-c_1(X),[\omega]), \mathcal{T}(-c_1(X),[\omega]) \right) 
> -\frac{3}{2} \alpha_X([\omega])
\end{equation}
is satisfied, where $$\sigma(-c_1(X),[\omega]) := \inf \{\delta : -c_1(X) - \delta[\omega] \leq 0 \}$$
and $\mathcal{T}(-c_1(X),[\omega])$ is defined in Lemma \ref{Lemma one over t}. 
\end{cor} 

\begin{proof}
Since we are working on a compact K\"ahler surface it follows from \cite[Lemma 4.1]{SongWeinkove} that there exists a constant $C > 0$ such that the entropy $\mathrm{H}_{\omega}$ satisfies the inequality
$$
\mathrm{H}_{\omega}(\varphi) \geq \frac{3}{2}\alpha_X([\omega])(\mathrm{I}_{\omega} - \mathrm{J}_{\omega})(\varphi) - C
$$
for all $\varphi \in \mathcal{H}$. As an immediate consequence of the Chen-Tian formula \cite{Chen2000} it moreover follows that $[\omega]$ admits a constant scalar curvature K\"ahler metric whenever
$$
\Gamma_{-\mathrm{Ric}(\omega)}^{\mathrm{pp}}(\omega) + \frac{3}{2}\alpha_X([\omega]) > 0. 
$$
Now note that if $\Gamma_{-\mathrm{Ric}(\omega)}^{\mathrm{pp}}(\omega) \geq \mathcal{T}(-c_1(X),[\omega]),$ then the desired conclusion holds, despite the fact that Theorem \ref{Thm main intro surfaces} does not apply. On the other hand, if $\Gamma_{-\mathrm{Ric}(\omega)}^{\mathrm{pp}}(\omega) < \mathcal{T}(-c_1(X),[\omega]),$ then by Theorem \ref{Thm main intro surfaces} the conclusion follows immediately from the above.
\end{proof}

\noindent We need to consider the minimum in \eqref{Equation min equation cscK} because of the eventuality that $c_1(X)$ is not negative. Indeed, the sufficient condition for existence in Corollary \ref{Corollary existence of cscK criterion} should be compared to other properness criteria using the alpha invariant, see for instance \cite{FLSW, LiShiYao, Dervanalphainvariant}, as well as \cite[Corollary 1.5]{ChenChengII} and references therein. The main improvement is expected to happen when $X$ is Fano or $c_1(X)$ has no sign. Higher dimensional analogues may also be obtained, using Theorem \ref{Thm formula higher dimension}.

\subsection{Remarks on the higher dimensional case} 
We finally discuss the possibilities of finding examples also of higher dimensional manifolds that are perfect (Definition \ref{Definition perfect}), i.e. admit solutions to the $J_{\theta,\omega}$-equation for every pair of K\"ahler forms $(\theta,\omega)$ on $X$. Following the Lejmi-Sz\'ekelyhidi conjecture (Conjecture \ref{Conjecture Collins-Szekelyhidi}) there are more conditions, compared to the surface case, that potentially obstruct such a property from being true. The next result shows that there are indeed serious restrictions on higher dimensional perfect manifolds:

\begin{prop} Suppose that $X$ is a perfect manifold (Definition \ref{Definition perfect}) for which the Lejmi-Sz\'ekelyhidi conjecture holds. Then every $2$-dimensional subvariety $V \subset X$ is also perfect, i.e. satisfies $\mathrm{Big}_V = \mathcal{C}_V$. 
\end{prop}

\begin{proof}
Suppose that $V \subset X$ is a submanifold of $X$ of dimension $p =2$, and assume that it is not perfect. Then by Proposition \ref{Theorem application 1} it follows that $\mathrm{Big}_V = \mathcal{C}_V$, and moreover we have a complete description of what K\"ahler classes admit a solution to the J-equation \eqref{Equation J}, by Theorem \ref{Theorem application 3}. In particular, it follows that we can find a K\"ahler form coming from $X$, i.e. $\chi_V := \omega_{\vert V}$ for some K\"ahler form $\omega$ on $X$, for which we cannot find a solution to the $\mathrm{J}_{\theta,\omega}$-equation for all smooth closed $(1,1)$-forms $\theta$ on $X$. In fact, by the Lejmi-Sz\'ekelyhidi conjecture and Theorem \ref{Thm formula higher dimension} combined with the above, there is then a subvariety $W \subset V$ of dimension $0 \leq \mathrm{dim}V \leq \mathrm{dim} V - 1$ such that
$$
C_{\theta,\omega} \int_W \omega^p - p\int_W \omega^{p-1} \wedge \theta \leq 0.
$$ 
But since $W \subset V$ is also a subvariety of $X$, Theorem \ref{Thm formula higher dimension} implies that we cannot find a solution to the $\mathrm{J}$-equation in $[\omega] \in \mathcal{C}_X$ either. Hence $X$ is not perfect. Taking the contrapositive, we see that if $X$ is perfect, then every $2$-dimensional subvariety $V \subset X$ is perfect.
\end{proof}

\subsection{Example I: Computation of the optimal constant for products of smooth curves} \label{Section examples}
In order to give an explicit example we revisit J. Ross' computations for products of smooth irreducible projective curves $X = C \times C$ of genus $g \geq 2$, see \cite{Rossproductofcurves}. In order to recall the setup, let for $i = 1,2$, $\pi_i$ be the projection onto the $i^{th}$ factor, and let $f_i$ be the class of $\pi^{-1}(p)$ in $N^1(X)_{\mathbb{Q}}$ (this is independent of the point $p \in X$ chosen). Let moreover $\delta$ be the class of the diagonal $\Delta \subset C \times C$. As described in \cite{Lazarsfeld} the canonical bundle of $X$ can then be written $K = (2g - 2)(f_1 + f_2)$, and it is ample. In his paper Ross considers the $\mathbb{Q}$-divisor
$$
L_t := tf - \delta',
$$
where $f := f_1 + f_2$ and $\delta' := \delta - f$. They point out that $L_t$ is ample for all large enough $t > 0$, and write 
$$
s_C := \inf \{ t: L_t \; \text{ample} \}.
$$ 
In \cite[Theorem 3.3]{Rossproductofcurves} it is shown that if $C$ admits a branched cover $\pi: C \rightarrow \mathbb{P}^1$ of degree $d$ with $2 \leq d - 1 \leq \sqrt{g}$ (such curves exist if $g \geq 5$, as shown in \cite[Theorem 4.4]{Rossproductofcurves}), then $X = C \times C$ is not slope semistable with respect to $L_t$ for $t > s_C$ sufficiently close to $s_C$. In particular, as $t \rightarrow s_C$ the stability threshold $\Gamma(c_1(L_t))$ corresponding to the Mabuchi K-energy becomes negative. 

It is interesting to note how the corresponding pluripotential stability threshold $\Gamma_{K}^{\mathrm{pp}}(c_1(L_t))$ behaves as $t \rightarrow s_C$ (here the notation should be interpreted as in Remark \ref{Remark cohomology notation} and Notation \ref{Remark notation}). To study this we explicitly compute the optimal properness constant using Theorem \ref{Thm main intro surfaces}. To do this, we first need to compute the infimum of all $\delta > 0$ such that $K - \delta L_t < 0$, or equivalently such that $\delta L_t - K$ is ample. We have
$$
\delta L_t - K = \delta L_t - (2g-2)f = \delta \left[(t - \frac{2g - 2}{\delta})f - \delta' \right]
$$
and by definition this is ample if and only if
$$
t - \frac{2g - 2}{\delta} > s_C
$$ 
i.e. precisely when
$$
\delta < \frac{2g -2}{t - s_C}.
$$
In other words, we have 
$$
\inf \{\delta > 0 \; \vert \; K - \delta L_t < 0 \} = \frac{2g -2}{t - s_C}.
$$
Moreover it follows from \cite{Lazarsfeld, Rossproductofcurves} that $K \cdot L_t = 2t(2g-2)$ and $L_t^2 = 2(t^2 - g)$. Hence by Theorem \ref{Thm main intro surfaces} the formula for the optimal threshold constant is 
$$
\Gamma_{K}^{\mathrm{pp}}(L_t) = \frac{2t(2g-2)}{t^2 - g} - \frac{2g -2}{t - s_C}.
$$
It is immediately clear from this that if $s_C > \sqrt{g}$, then the above threshold tends to minus infinity as $t \rightarrow s_C$. In the slope unstable examples of J. Ross he chose $C$ that admits a branched cover $\pi: C \rightarrow \mathbb{P}^1$ of degree $d$ with $2 \leq d - 1 \leq \sqrt{g}$. As shown in \cite{Kouvidakis} then $s_C = \frac{g}{d-1}$ and we may explicitly compute that 
$$
\Gamma_{K}^{\mathrm{pp}}(L_t) = \frac{2t(2g-2)}{t^2 - g} - \frac{(d-1)(2g -2)}{(d-1)t - g} 
$$
which tends to minus infinity as $t \rightarrow s_C$. 

If instead $s_C = \sqrt{g}$, then we compute 
$$
\Gamma_{K}^{\mathrm{pp}}(L_t) = (2g-2)\frac{1}{t + \sqrt{g}} \rightarrow \frac{g - 1}{\sqrt{g}},
$$
as $t \rightarrow s_C$. For example if $g = 4$ then we get
$
\Gamma_{K}^{\mathrm{pp}}(L_t) = \frac{6}{t + 2},
$
which is always positive, and increasing as $t \rightarrow 2 = \sqrt{g}$. By inspecting the sign of the above expression we can recover a simple proof of \cite[Lemma 4.3]{Rossproductofcurves}. 

The observation that the above threshold tends to minus infinity close to the boundary whenever $s_C = \frac{g}{d-1}$ is consistent with Theorem \ref{Theorem application 3} and the fact that in this situation $X$ admits a curve $Z$ of negative self-intersection (see \cite[Lemma 3.1]{Rossproductofcurves}).

\subsection{Example II: Blowups of Calabi dream manifolds} \label{Example blowup}
We here use the cohomological notation in the sense of Remark \ref{Remark cohomology notation} and Notation \ref{Remark notation}: Let $X$ be any Calabi dream manifold with discrete automorphism group (many examples of these have appeared in \cite{ChenChengII}, see also the discussion in this paper). Consider the blowup
$
\pi_p: Y_p \rightarrow X
$
of $X$ at any point $p \in X$. Let $E$ be the exceptional divisor and let $[\omega] \in \mathcal{C}_X$ be a K\"ahler class on $X$. By Arezzo-Pacard \cite{ArezzoPacard} it follows that $Y_p$ admits a constant scalar curvature K\"ahler metric in every K\"ahler class of the form
$$
\gamma_{\epsilon} := \pi^*[\omega] - \epsilon[E]
$$
for $\epsilon > 0$ small enough. In other words, we can find K\"ahler classes $\gamma_{\epsilon}$ arbitrarily close to the boundary class $\pi^*[\omega] \in \partial \mathcal{C}_X$ in which the constant scalar curvature equation can be solved. On the other hand, $Y := Y_p$ contains a $(-1)$-curve, so by Corollary \ref{Corollary unstable example} it follows that $\Gamma^{\mathrm{pp}}_{- c_1(Y)}(\gamma_{\epsilon})$ tends to minus infinity as $\epsilon \rightarrow 0$. It is interesting to note that for any blowup of a Calabi dream manifold as above, the entropy term compensates sufficiently in order to make the Mabuchi K-energy proper. 

A similar phenomenon can also be observed in light of the main result of \cite{JianShiSong}. 

\bigskip 

\section{Applications part II: Algebraic stability notions, openness and optimal degenerations}

\medskip

\subsection{Analytic and algebraic thresholds and existence of optimal degenerations} \label{Section optimal degeneration}

\noindent In Section \ref{Section threshold properties} we have established a precise description of how the stability threshold $\Gamma_{\theta}^{\mathrm{pp}}(\omega)$ behaves under variation of $[\theta]$, and we gave in our main result an explicit formula for the variation also in $[\omega] \in \mathcal{C}_X$ as long as we restrict to surfaces.  Extending such a result to higher dimension would be important for applications. Even without an explicit formula, it is fascinating to try to understand if the stability threshold varies continuously also in higher dimension. 
By combining our results with the assumption that the Lejmi-Sz\'ekelyhidi conjecture holds (see e.g. \cite{GaoChen}), we can address this question.  
More precisely, the linearity of the stability threshold functions can be used to compare the analytic threshold $\Gamma_{\theta}^{\mathrm{pp}}(\omega)$ and the algebraic stability threshold associated to uniform J-stability (which is known to be upper semi-continuous as we vary $[\omega] \in \mathcal{C}_X$, see \cite{SD3}). This will use the formalism in e.g. \cite{LejmiGabor, SD1, DervanRoss, SD2, Dervanrelative, GaoChen} to which we refer for details and full definitions, see also \cite{Berman,BHJ2,LejmiGabor} in the projective case. 

To state our results we use the non-projective formalism for test configurations that was introduced in \cite{SD1,DervanRoss}, which in the projective case coincides with the notion of Lejmi-Sz\'ekelyhidi \cite{LejmiGabor}. In this terminology, consider the norm of test configuration given by 
$$
||(\mathcal{X},\mathcal{A})|| := \lim_{t \rightarrow +\infty} t^{-1} \mathrm{E}_{\omega}^{\omega}(\varphi_t),
$$
where $(\mathcal{X},\mathcal{A})$ is a test configuration for $(X,[\omega])$ in the sense of \cite{SD1, DervanRoss}, and $(\varphi_t)$ is the associated geodesic ray (in the sense of \cite{SD1, DervanRoss}). Write moreover
\begin{equation} \label{Equation asymptotic}
\mathbf{E}^{\theta}_{\omega}(\mathcal{X},\mathcal{A}) := \lim_{t \rightarrow +\infty} t^{-1} \mathrm{E}_{\omega}^{\theta}(\varphi_t)
\end{equation}
for the radial energy functional corresponding to the pair $([\theta],[\omega]) \in \mathcal{C}_X \times \mathcal{C}_X$ (by e.g. \cite[Theorem B]{SD1} the non-Archimedean functional $\mathbf{E}^{\theta}_{\omega}$ is an intersection number and is independent of the choice of representative $\theta \in [\theta]$ and $\omega \in [\omega]$). Note that it is well-defined by convexity, whenever both $[\theta]$ and $[\omega]$ are K\"ahler classes. Define also an \emph{algebraic stability threshold} by 
$$
\Delta^{\mathrm{pp}}_{\theta}(\omega) := \sup \{ \delta \in \mathbb{R} \; \vert \; \mathbf{E}^{\theta}_{\omega}(\mathcal{X},\mathcal{A}) \geq \delta ||(\mathcal{X},\mathcal{A})|| \} 
$$
$$
= \inf_{||(\mathcal{X},\mathcal{A})|| = 1} \mathbf{E}^{\theta}_{\omega}(\mathcal{X},\mathcal{A})
$$
for all test configurations $(\mathcal{X},\mathcal{A})$ for $(X,[\omega])$. 
We say that $(X,\omega)$ is uniformly J-stable if $\Delta^{\mathrm{pp}}_{\theta}(\omega) > 0$, and let it be understood that the notion considered depends also on the underlying class $[\theta] \in \mathcal{C}_X$ that we fix. 

We now add the observation that if the Lejmi-Sz\'ekelyhidi conjecture holds (cf. \cite{GaoChen}), then not only is J-stability equivalent to properness of the relevant functional, but the respective algebraic and analytic thresholds in fact coincide (which in turn strongly suggest that the thresholds vary continuously as a function of $([\theta],[\omega]) \in \mathcal{C}_X \times \mathcal{C}_X$):  

\begin{thm} \label{Thm thresholds coincide}
Let $X$ be a compact K\"ahler manifold with discrete automorphism group for which the Lejmi-Sz\'ekelyhidi conjecture holds \emph{(e.g. $n = 2$ or $X$ toric)}. Suppose that $(\theta,\omega)$ are K\"ahler forms on $X$ such that $\Gamma_{\theta}^{\mathrm{pp}}(\omega) < \mathcal{T}([\theta],[\omega])$. Then the analytic and algebraic stability thresholds coincide, i.e.
$$ 
\Delta_{\theta}^{\mathrm{pp}}(\omega) = \Gamma_{\theta}^{\mathrm{pp}}(\omega).
$$
In particular, we have the formula
$$
\Delta^{\mathrm{pp}}_{\theta}(\omega) = \inf_{V} \frac{ C_{\theta,\omega} \int_V \omega^p - p\int_V \omega^{p-1} \wedge \theta}{{(n-p) \int_V \omega^p}},
$$ 
where the infimum is taken over all subvarieties $V \subset X$ of dimension $p \leq n - 1$. 
\end{thm}

\begin{proof}
Along the same lines of the proofs of previous sections, let $\theta_s$ be any K\"ahler form on $X$ such that $[\theta_s] := (1-s)[\theta] + s[\omega]$. Consider the quantities
$$
P(s) := \Gamma_{\theta}^{\mathrm{pp}}(\omega) \; , \; Q(s) := \inf_{||(\mathcal{X},\mathcal{A})|| = 1} \mathbf{E}^{\theta}_{\omega}(\mathcal{X},\mathcal{A})  
$$
which are both linear for $s \leq 1$. It follows from Lemma \ref{Lemma piecewise linear} and \eqref{Equation asymptotic} that $P(1) = Q(1) = 1$. Moreover, by \cite{GaoChen} both $P(s) > 0$ and $Q(s) > 0$ are equivalent to existence of a solution to the J-equation. Therefore, as before, if there exists a K\"ahler form $\theta$ on $X$ for which the $J_{\theta,\omega}$-equation \eqref{Equation J} is \emph{not} solvable, then $P(s_0) = Q(s_0)$ also for some $s_0 \in [0,1]$ with $s_0 \neq 1$. 
The equality of the two thresholds extends to $\beta_s$ for all $s \leq 1$ by linearity. Taking $s = 0$ finishes the proof.
\end{proof}

\begin{rem} \label{Section sub continuity} \emph{(Applications to continuity and openness of J-stability) It is interesting to discuss implications of the above Theorem \ref{Thm thresholds coincide} to the study continuity properties of the stability thresholds as we vary $[\theta]$ and $[\omega]$. Continuity is an immediate consequence of \cite{Chen2000} in the case of surfaces, and in general it was observed in \cite{SD3} that $\Delta_{\theta}^{\mathrm{pp}}(\omega)$ is upper semi-continuous. On the other hand, it is expected that $\Gamma_{\theta}^{\mathrm{pp}}(\omega)$, viewed as a cohomological quantity in the sense of Remark \ref{Remark cohomology notation} and Notation \ref{Remark notation}, is lower semi-continuous. Putting this information together would lead us to expect that both thresholds are in fact continuous, at least when our main results (Theorem \ref{Thm main intro surfaces} and \ref{Thm formula higher dimension intro}) hold, i.e. on the open subset 
$$
\mathcal{V}_X := \{([\theta],[\omega]) \in H^{1,1}(X,\mathbb{R}) \times \mathcal{C}_X \; : \; \Gamma_{\theta}^{\mathrm{pp}}(\omega) < \mathcal{T}([\theta],[\omega]) \}
$$ 
of $H^{1,1}(X,\mathbb{R}) \times \mathcal{C}_X$. 
This would in turn give a direct and simple proof of openness of uniform J-stability as we vary the underlying K\"ahler classes.}
\end{rem}

\noindent The continuity of stability threshold is a useful piece of information, and one may ask if the stability threshold is in fact always a rational function in the underlying K\"ahler classes, as is the case for surfaces (in the sense of Section \ref{Section main result surfaces}, see also Section \ref{Section examples}). This explicit rational expression for the stability threshold, valid in particular for all J-unstable classes, can moreover be used to add a last piece of information to the characterization of the subcone $\mathrm{Js}_X \subset \mathcal{C}_X$ of J-stable classes in the K\"ahler cone, with reference to any fixed $[\theta] \in \mathcal{C}_X$. For the purpose of stating our main results, let us recall that $(X,\omega)$ is J-semistable if and only if $\Delta_{\theta}^{\mathrm{pp}}(\omega) \geq 0$, depending only on the cohomology classes $[\theta]$ and $[\omega]$, and write $\mathrm{Jss}_X\subset \mathcal{C}_X$ for the subcone of J-semistable classes. If $X$ is a compact K\"ahler surface then both Lemma \ref{Lemma piecewise linear} and Theorem \ref{Thm thresholds coincide} apply, and hence we confirm the expectation that the J-semistable locus is precisely the closure of the (uniformly) J-stable locus: 

\begin{cor} \label{Cor closure} Let $X$ be a compact K\"ahler surface with discrete automorphism group. Then the J-semistable locus equals the closure of the uniformly J-stable locus in the K\"ahler cone, i.e. 
$$
\mathrm{Jss}_X = \overline{\mathrm{Js}}_X
$$
in $\mathcal{C}_X$. 
\end{cor}

\noindent In particular it is immediately seen from this that uniform J-stability is equivalent to the usual J-stability notion. It would be interesting to extend this result to the general higher dimensional case.  

Finally, we use the established piecewise linearity of the stability threshold to show existence of a minimizing test configuration for the algebraic threshold, assuming that J-stability is equivalent to uniform J-stability:

\begin{thm} \label{Theorem optimal test config} 
Suppose that $X$ is a compact K\"ahler manifold with discrete automorphism group such that every J-stable class is uniformly J-stable \emph{(e.g. n = 2)}. If $\Delta_{\theta}^{\mathrm{pp}}(\omega) < \mathcal{T}([\theta],[\omega])$, then 
the infimum is achieved by a minimizing test configuration $(\mathcal{X},\mathcal{A})_{min}$ of norm $1$ for $(X,[\omega])$, i.e. 
$$
\Delta_{\theta}^{\mathrm{pp}}(\omega) = \mathbf{E}^{\theta}_{\omega}(\mathcal{X},\mathcal{A})_{min}.
$$
\end{thm}

\begin{proof}
With the same setup as above, let $\theta_s$ be any smooth closed $(1,1)$-form on $X$ such that $[\theta_s] := (1-s)[\theta] + s[\omega]$, with $[\theta], [\omega] \in \mathcal{C}_X$, and consider the quantity
$$
Q(s) := \inf_{||(\mathcal{X},\mathcal{A})|| = 1} \mathbf{E}^{\theta_s}_{\omega}(\mathcal{X},\mathcal{A})  
$$
which is linear for $s \in [0,1]$ (this may be established by the exact same proof as in Lemma \ref{Lemma piecewise linear}, using also \eqref{Equation asymptotic}). 
In order to establish existence of a minimizer, assume that there exists an $s_0 \in (0,1)$ such that $Q(s) \leq 0$ for all $s \in [0,s_0]$. By continuity we then have $P(s_0) = Q(s_0) = 0$. Moreover, if J-stability and uniform J-stability are equivalent notions, then there exists a test configuration $(\tilde{\mathcal{X}},\tilde{\mathcal{A}})$ for which 
$$
\mathbf{E}^{\theta}_{\omega}(\tilde{\mathcal{X}},\tilde{\mathcal{A}}) = 0.
$$
This test configuration then minimizes $Q(s)$ for all $s \in [0,1]$, by linearity. 
\end{proof}

\noindent Provided that the Lejmi-Sz\'ekelyhidi conjecture holds (see \cite{GaoChen}) it follows from Theorem \ref{Thm inf achieved} that the above test configuration $(\tilde{\mathcal{X}},\tilde{\mathcal{A}})$ can be taken to be the slope test configuration where $\mathcal{X}$ is the deformation to the normal cone of the minimizing subvariety $V_{min} \subset X$. 

\medskip

\subsection{Acknowledgements}

I would like to thank Claudio Arezzo, Tristan Collins, Tam\'as Darvas, Ruadha\'i Dervan, Lothar G\"ottsche, Robert Berman, Julien Keller, Jacopo Stoppa and Roberto Svaldi for discussions related to this work. I also extend my gratitude to the anonymous referee for many helpful comments, and to ICTP Trieste, for providing excellent working conditions during the course of the project.

\end{document}